\documentclass[oneside,english,american]{amsart}
\usepackage{lmodern}

\usepackage[T1]{fontenc}
\usepackage[latin9]{inputenc}
\usepackage{geometry}
\usepackage{color}
\usepackage{babel}
\usepackage{booktabs}
\usepackage{units}
\usepackage{amsthm}
\usepackage{amsmath}
\usepackage{amssymb}
\usepackage{esint}
\usepackage[numbers]{natbib}
\usepackage[unicode=true,
 bookmarks=true,bookmarksnumbered=false,bookmarksopen=false,
 breaklinks=false,pdfborder={0 0 0},backref=false,colorlinks=true]
 {hyperref}
\hypersetup{
 linkcolor=Dandelion, citecolor=LimeGreen, urlcolor=JungleGreen}
 
\makeatletter

\providecommand{\tabularnewline}{\\}

\numberwithin{equation}{section}
\numberwithin{figure}{section}

\theoremstyle{plain}
\newtheorem{theorem}{\protect\theoremname}
  \theoremstyle{plain}
  \newtheorem{lem}[theorem]{\protect\lemmaname}
  \theoremstyle{plain}
  \newtheorem{prop}[theorem]{\protect\propositionname}
  \theoremstyle{remark}
  \newtheorem{rem}[theorem]{\protect\remarkname}
  \theoremstyle{definition}
  \newtheorem{defn}[theorem]{\protect\definitionname}
  \theoremstyle{plain}
  \newtheorem{cor}[theorem]{\protect\corollaryname}
  \theoremstyle{definition}
  \newtheorem{example}[theorem]{\protect\examplename}

\numberwithin{theorem}{section}
\usepackage{euscript}
\usepackage[usenames, dvipsnames]{xcolor}

\makeatother

  \addto\captionsamerican{\renewcommand{\corollaryname}{Corollary}}
  \addto\captionsamerican{\renewcommand{\definitionname}{Definition}}
  \addto\captionsamerican{\renewcommand{\examplename}{Example}}
  \addto\captionsamerican{\renewcommand{\lemmaname}{Lemma}}
  \addto\captionsamerican{\renewcommand{\propositionname}{Proposition}}
  \addto\captionsamerican{\renewcommand{\remarkname}{Remark}}
  \addto\captionsamerican{\renewcommand{\theoremname}{Theorem}}
  \addto\captionsenglish{\renewcommand{\corollaryname}{Corollary}}
  \addto\captionsenglish{\renewcommand{\definitionname}{Definition}}
  \addto\captionsenglish{\renewcommand{\examplename}{Example}}
  \addto\captionsenglish{\renewcommand{\lemmaname}{Lemma}}
  \addto\captionsenglish{\renewcommand{\propositionname}{Proposition}}
  \addto\captionsenglish{\renewcommand{\remarkname}{Remark}}
  \addto\captionsenglish{\renewcommand{\theoremname}{Theorem}}
  \providecommand{\corollaryname}{Corollary}
  \providecommand{\definitionname}{Definition}
  \providecommand{\examplename}{Example}
  \providecommand{\lemmaname}{Lemma}
  \providecommand{\propositionname}{Proposition}
  \providecommand{\remarkname}{Remark}
\providecommand{\theoremname}{Theorem}

\begin{document}
\title{Volume comparison for $\mathcal{C}^{1,1}$-metrics}

\author{Melanie Graf}

\date{02.04.2016}

\address{Faculty of Mathematics, University of Vienna}

\email{melanie.graf@univie.ac.at}
\begin{abstract}
The aim of this paper is to generalize certain volume comparison theorems (Bishop-Gromov and a recent result of Treude and Grant \citep{TG}) for smooth Riemannian or Lorentzian manifolds to metrics that are only $\mathcal{C}^{1,1}$ (differentiable with Lipschitz continuous derivatives). 

In particular we establish (using approximation methods) a volume monotonicity result for the evolution of a compact subset of a spacelike, acausal, future causally complete (i.e., the intersection of any past causal cone with the hypersurface is relatively compact) hypersurface with an upper bound on the mean curvature in a  globally hyperbolic spacetime with a  $\mathcal{C}^{1,1}$-metric with a lower bound on the timelike Ricci curvature, provided all timelike geodesics starting in this compact set exist long enough. As an intermediate step we also show that the cut locus of such a hypersurface still has measure zero in this regularity -- generalizing the well-known result for smooth metrics.

To show that these volume comparison results have some very nice applications we then give a proof of Myers' theorem, of a simple singularity theorem for globally hyperbolic spacetimes, and of Hawking's singularity theorem directly in this regularity.

\end{abstract}



\maketitle

\section{Introduction}

There are many similarities between the ideas used in the proof of
Riemannian comparison theorems (in particular Myers' theorem) and
the singularity theorems in Lorentzian geometry. Both use curvature
conditions to obtain that in some sense the maximal length of a geodesic
without conjugate points is bounded: in the case of Myers' theorem
one assumes completeness and obtains a bound on the
diameter of the manifold (as the distance between two points is given by the length
of a minimizing geodesic, which can not have conjugate points) and
in the case of, e.g., the Hawking singularity theorem the assumptions
together with geodesic completeness would imply compactness of a certain
Cauchy horizon which then gives a contradiction. 
While there has been some interest in developing Lorentzian analogues to many results from Riemannian comparison geometry in general (see e.g.\ \citep{Alexander_Lorentz_sec_curv_bounds}, \citep{andersson_howard} and \citep{ehrlich_sanchez})
this close connection to the singularity theorems
was explored further 
by Treude and Grant in their recent paper \citep{TG}, where they use Riccati comparison techniques to
prove area and volume monotonicity theorems
in Lorentzian geometry (with respect to fixed Lorentzian warped product
manifolds). These are then applied to give a new proof of the classical Hawking
singularity theorem. 

We will show that many of these results carry over to
$\mathcal{C}^{1,1}$ (locally Lipschitz continuous first derivatives)
regularity by showing volume monotonicity results for both Riemannian
and Lorentzian $\mathcal{C}^{1,1}$-metrics with appropriate curvature
bounds and applying them to prove a version of Myers' theorem and
Hawking's singularity theorem, respectively.

In general, for a (semi-)Riemannian metric the class $\mathcal{C}^{1,1}$
is the lowest differentiability class of the metric where one still
has local existence and uniqueness of solutions of the geodesic equation.
Also by Rademacher's theorem all curvature terms still exist almost
everywhere and are locally bounded, which allows the definition of
curvature bounds in the following way. We say that the Ricci curvature
tensor $\mathbf{Ric}$ is bounded from below (by $\kappa$) if for
every smooth, local vector field $X\in\mathfrak{X}(U)$ for some open
and relatively compact $U\subset M$ one has that the function
\begin{equation}
p\mapsto\mathbf{Ric}(p)(X_{p},X_{p})-(n-1)\kappa g(p)(X_{p},X_{p})\label{eq:ric est}
\end{equation}
is non-negative as an element of $L^{\infty}(U)$ (i.e., is non-negative almost
everywhere). If $M$ is Lorentzian we say that the timelike Ricci
curvature is bounded from below (by $-\kappa$) if the above holds for any smooth,
local timelike vector field. Clearly this coincides with the usual
notion for smooth metrics.

As further motivation for studying metrics of this regularity we give
a brief overview about the specific situations in the Riemannian and
the Lorentzian setting.

In Riemannian geometry there are ways to generalize curvature bounds
to even lower regularity, however this requires - at first glance
- very different definitions (see e.g.\ \citep{sturm_metricMeasureII,LottVillani2009},
where metric measure spaces with lower bounds on the Ricci curvature
are studied). While these definitions are equivalent for smooth metrics
this has not yet been shown for $\mathcal{C}^{1,1}$-metrics, so at
least for now those two approaches are independent.

In Lorentzian geometry there has recently been an increased interest
and many advances in the understanding of low regularity spacetimes
(i.e.\ $\mathcal{C}^{1,1}$- instead of $\mathcal{C}^{2}$-metrics,
see \citep{CG,Minguzzi_convexnbhdsLipConSprays,KSS,KSSV}), which
allowed the proof of both the Hawking and the Penrose singularity
theorem in this regularity (see \citep{hawkingc11,PenroseC11inpaper}),
a problem that had been open for a long time (cf.\ \citep{Seno1}).
From the viewpoint of general relativity, the importance of this regularity
is that it allows for a finite jump in the matter variables via the
Einstein equations. It is also worth noting that many of the standard results fail dramatically when lowering the regularity further, for example it is shown in \citep{CG} that for
any $\alpha\in(0,1)$ there exist \textquoteleft{}bubbling metrics\textquoteright{}
(of regularity $\mathcal{C}^{0,\alpha}$), whose lightcones have nonempty
interior.

The plan of the paper is as follows. In section \ref{sec:Volume-comparison-for-Riemannian}
we study Riemannian manifolds with $\mathcal{C}^{1,1}$-metrics with
a lower bound on the Ricci curvature and show a $\mathcal{C}^{1,1}$
version of the Bishop-Gromov volume comparison theorem for Riemannian
manifolds with a lower bound on the Ricci curvature. This also serves
as a preparation for the Lorentzian case as it requires significantly
less technical details but the ideas remain largely the same. In section
\ref{sec:The-Lorentzian-case} we first give the definition of the
cosmological comparison condition (as introduced in \citep{TG}) and
a brief overview of relevant results from causality theory for $\mathcal{C}^{1,1}$-metrics,
in particular concerning global hyperbolicity and maximizing geodesics
to a subset. Then we show the existence of suitable approximating
metrics (using results from \citep{CG,KSSV,hawkingc11}) and in section
\ref{sub:The-cut-locus} we show that for $\mathcal{C}^{1,1}$-metrics
 the cut locus still has measure zero. As a last preparation we define
our comparison spacetimes (again introduced in \citep{TG}) as Robertson-Walker spacetimes
with constant Ricci curvature and study their dependence on the curvature
quantities $\kappa$ and $\beta$. This then allows us to show (as
a generalization of \citep[Thm.~9]{TG} to $\mathcal{C}^{1,1}$-metrics)
\begin{theorem}
[Volume comparison]\label{thm: volume}Let $\kappa,\beta\in\mathbb{R}$,
$g\in\mathcal{C}^{1,1}$ and assume $\left(M,g,\Sigma\right)$ is
globally hyperbolic and satisfies $CCC(\kappa,\beta)$ (see Def.\
\ref{def: CCC}). Let $A\subset\Sigma$ be compact with $\mu_{\Sigma}(\partial A)=0$,
$B\subset\Sigma_{\kappa,\beta}$ (with finite, non-zero area) and
$T>0$ such that all timelike, future directed, unit speed geodesics starting
orthogonally to $A$ exist until at least $T$. Then the function
\[
t\mapsto\frac{\mathrm{vol}\, B_{A}^{+}(t)}{\mathrm{vol}_{\kappa,\beta}B_{B}^{+}(t)}
\]
is nonincreasing on $\left[0,T\right]$.
\end{theorem}
Finally, in section \ref{sec:Applications}, as applications we give
a proof of a $\mathcal{C}^{1,1}$-Myers' theorem in the Riemannian
and two $\mathcal{C}^{1,1}$-singularity theorems (one of them being
an alternative proof of the $\mathcal{C}^{1,1}$-version of Hawking's
theorem proved in \citep[Thm.~1.1]{hawkingc11}) in the Lorentzian
case.

\subsection*{Notation}

Throughout $M$ will always be a connected, Hausdorff and second countable
smooth manifold of dimension $n\geq2$. For a semi-Riemannian metric
$g$ on $M$ the curvature tensor of the metric is defined with the
convention $R(X,Y)Z=\left(\left[\nabla_{X},\nabla_{Y}\right]-\nabla_{[X,Y]}\right)Z$
and we denote the Ricci tensor of $g$ by $\mathbf{Ric}$.

\section{\label{sec:Volume-comparison-for-Riemannian}Volume comparison for
Riemannian \texorpdfstring{$\mathcal{C}^{1,1}$}{C1,1}-metrics}

The goal of this first section is to show a $\mathcal{C}^{1,1}$ version
of the Bishop-Gromov volume comparison theorem.
\begin{theorem}
[Bishop-Gromov] \label{thm:classical bishop gromov}Suppose $\left(M,g\right)$
(with $g$ smooth) is a complete Riemannian manifold with $\mathbf{Ric}\geq(n-1)\kappa g$
for some $\kappa\in\mathbb{R}$. Then
\[
r\mapsto\frac{\mathrm{vol}B_{p}(r)}{\mathrm{vol}_{\kappa}B^{\kappa}(r)},
\]
where $B^{\kappa}(r)$ denotes any ball of radius $r$ in the $n$-dimensional simply-connected
Riemannian manifold with constant sectional curvature equal to $\kappa$,
is a nonincreasing function on $(0,\infty)$ and $\mathrm{vol}B_{p}(r)\leq\mathrm{vol}_{\kappa}B^{\kappa}(r)$.
\end{theorem}
A proof of the classical result (for smooth metrics) can be found,
e.g., in \citep[Cor.~3.3]{zhu1997comparison}. The idea of the proof
for $\mathcal{C}^{1,1}$-metrics is to apply the classical result
to some smooth approximating metrics, so we first have to show that
we can find approximations such that $\left(M,g_{\varepsilon}\right)$
is a complete Riemannian manifold and that for any compact $K\subset M$
and $\delta>0$ we have $\mathbf{Ric}_{\varepsilon}|_{K}\geq\left(n-1\right)(\kappa-\delta)g_{\varepsilon}|_{K}$
(where $\mathbf{Ric}_{\varepsilon}$ denotes the Ricci tensor of $g_{\varepsilon}$)
for $\varepsilon$ small enough. 
\begin{lem}
\label{lem:riem approx}Let $g\in\mathcal{C}^{1,1}$ be a (geodesically) complete
Riemannian metric on $M$. Then there exist smooth complete Riemannian
metrics $g_{\varepsilon}$ on $M$ such that $g_{\varepsilon}\to g$
in $\mathcal{C}^{1}$, the approximations have locally uniformly bounded
second derivatives and 
\begin{equation}
d(g,g_{\varepsilon}):=\sup_{p\in M}\sup_{0\neq X,Y\in T_{p}M}\frac{\left|g(X,Y)-g_{\varepsilon}(X,Y)\right|}{\left|X\right|_{g}\left|Y\right|_{g}}\to0.\label{eq:d(g,geps) to zero}
\end{equation}
\end{lem}
\begin{proof}
It is well known that one can construct smooth, symmetric $\left(0,2\right)$-tensor
fields $\tilde{g}_{\varepsilon}\in\mathcal{T}_{2}^{0}(M)$ with $\tilde{g}_{\varepsilon}\to g$
in $\mathcal{C}^{1}$ and locally uniformly bounded second derivatives
by gluing together componentwise convolutions via a partition of unity:
Let $(U_{\alpha},\psi_{\alpha})$ be a (countable) atlas and $\left\{ \chi_{\alpha}\right\} $
a partition of unity subordinate to the $U_{\alpha}$ and choose functions
$\zeta_{\alpha}\in\mathcal{C}^{\infty}(U_{\alpha})$ with compact
support in $U_{\alpha}$ such that $0\leq\zeta_{\alpha}\leq1$ and $\zeta_{\alpha}\equiv1$
on an open neighborhood of $\mathrm{supp}(\chi_{\alpha})$ in $U_{\alpha}$.
Given a locally integrable $(p,q)$-tensor field $T$ we set 
\begin{equation}
\tilde{T}_{\varepsilon}=\sum_{\alpha}\zeta_{\alpha}\cdot\psi_{\alpha}^{*}\left(\left(\tilde{\chi}_{\alpha}\, T^{\alpha}\right)*\rho_{\varepsilon}\right),\label{eq:iota from distribs to gen func on X}
\end{equation}
where $T^{\alpha}\in L_{\mathrm{loc}}^{1}\left(\psi_{\alpha}(U_{\alpha}),\mathbb{R}^{n^{p+q}}\right)$
denotes the chart representation of $T$, $\tilde{\chi}_{\alpha}:=\chi_{\alpha}\circ\psi_{\alpha}^{-1}$
and the convolution is to be understood componentwise. Note that
this construction also ensures that the map $\left(\varepsilon,p\right)\mapsto\tilde{g}_{\varepsilon}(p)$
is smooth. 

Now let $\delta>0$. By locally uniform convergence we get that
for any $K\subset M$ compact, w.l.o.g.\ $K\subset U_{\alpha}$ for
some chart domain $U_{\alpha}$ (otherwise we may cover $K$ by finitely
many of those), there exists $\varepsilon_{K}$ such that 
\begin{multline}
\sup_{p\in K}\sup_{0\neq X,Y\in T_{p}M}\frac{\left|g(X,Y)-\tilde{g}_{\varepsilon}(X,Y)\right|}{\left|X\right|_{g}\left|Y\right|_{g}}\leq\\
\leq\sup_{p\in K}\sup_{0\neq X,Y\in T_{p}M}\frac{\left\Vert \left(g_{ij}-\tilde{g}_{\varepsilon,ij}\right)X^{j}\right\Vert _{e}\,\left\Vert Y\right\Vert _{e}}{\left|X\right|_{g}\left|Y\right|_{g}}\leq nC^{2}\sup_{i,j\leq n}\sup_{p\in K}\left|g_{ij}(p)-\tilde{g}_{\varepsilon,ij}(p)\right|<\delta\label{eq:blubb}
\end{multline}
for all $\varepsilon\leq\varepsilon_{K}$ (here $\left\Vert .\right\Vert _{e}$
denotes the euclidean norm on $\mathbb{R}^{n}$ and we used Cauchy's
inequality, $\left\Vert AX\right\Vert _{e}\leq n\,\max_{i,j\leq n}|A_{ij}|\left\Vert X\right\Vert _{e}$
and that $\frac{\left\Vert X\right\Vert _{e}}{|X|_{g}}<C$, where
$C=\sup_{\{X\in TM|_{K}:|X|_{g}=1\}}\left\Vert X\right\Vert _{e}<\infty$,
for any $X\in TM|_{K}$). But then the globalization lemma \citep[Lem.~2.4]{KSSV}
allows us to construct (new) approximations $g_{\varepsilon}:p\mapsto\tilde{g}_{u(\varepsilon,p)}(p)$
such that for each compact set $K\subset M$ there exists $\varepsilon_{K}$
such that $g_{\varepsilon}(p)=\tilde{g}_{\varepsilon}(p)$ for all
$\varepsilon\leq\varepsilon_{K}$ and $p\in K$ (in particular the $g_{\varepsilon}$
still satisfy $g_{\varepsilon}\to g$ in $\mathcal{C}^{1}$ and have
locally uniformly bounded second derivatives) and such that for each
$\delta>0$ there exists $\varepsilon_{0}(\delta)$ such that $d(g,g_{\varepsilon})<\delta$
for all $\varepsilon\leq\varepsilon_{0}$, i.e., $d(g,g_{\varepsilon})\to0$.

It remains to show completeness and that the $g_{\varepsilon}$ are
Riemannian. This follows from (\ref{eq:d(g,geps) to zero}): For any
$\delta>0$ there exists $\varepsilon_{0}$ such that for all $\varepsilon\leq\varepsilon_{0}$
one has $\left|g(X,X)-g_{\varepsilon}(X,X)\right|<\delta g(X,X)$
for all $X\in TM$, $X\neq0$, hence 
\begin{equation}
\left(1-\delta\right)g(X,X)\leq g_{\varepsilon}(X,X)\leq\left(1+\delta\right)g(X,X).\label{eq:estimate for norm(X) in geps}
\end{equation}
From this it immediately follows that for $\varepsilon$ small enough
positive definiteness of $g$ implies positive definiteness of $g_{\varepsilon}$,
hence the approximations are Riemannian, and it also immediately gives
$\sqrt{1-\delta}L_{g}(\gamma)\leq L_{g_{\varepsilon}}(\gamma)\leq\sqrt{1+\delta}L_{g}(\gamma)$
for any (locally Lipschitz) curve $\gamma$. But this implies that
for $\varepsilon\leq\varepsilon_{0}$ we have 
\begin{equation}
\sqrt{1-\delta}\: d_{g}(p,q)\leq d_{g_{\varepsilon}}(p,q)\leq\sqrt{1+\delta}\ d_{g}(p,q)\label{eq: distance estimates}
\end{equation}
and thus $B_{\varepsilon,p}(r)\subset B_{p}(\frac{r}{\sqrt{1-\delta}})\subset \exp_{p}(\frac{r}{\sqrt{1-\delta}} \cdot \{v\in T_{p}M:\,\left|v\right|_{g}=1\})$
is relatively compact for all $p\in M$ and $r>0$, so $\left(M,g_{\varepsilon}\right)$
is a complete Riemannian manifold by the Hopf-Rinov theorem.
\end{proof}
The next Lemma deals with the Ricci curvature estimate and its proof
is largely analogous to the Lorentzian version shown in \citep[Lem.~3.2]{hawkingc11}
for $\kappa=0$, but a bit less involved.
\begin{lem}
\label{lem:Riem approx ricci}Let $g\in\mathcal{C}^{1,1}$ be a complete
Riemannian metric on $M$ that satisfies $\mathbf{Ric}\geq\left(n-1\right)\kappa g$.
Then there exist smooth approximations $g_{\varepsilon}$ with all
properties of the previous Lemma and such that for any compact $K\subset M$
and $\delta>0$ there exists $\varepsilon_{0}$ such that 
\[
\mathbf{Ric}_{\varepsilon}|_{K}\geq\left(n-1\right)\left(\kappa-\delta\right)g_{\varepsilon}|_{K}
\]
 for any $\varepsilon\leq\varepsilon_{0}$.\end{lem}
\begin{proof}
We first note that 
\begin{equation}
\mathbf{Ric}_{\varepsilon}-\mathbf{\tilde{Ric}}_{\varepsilon}\to0\quad\text{uniformly on compact sets},\label{eq:R-R*rho_eps goes to zero-1}
\end{equation}
where $\mathbf{\tilde{Ric}}_{\varepsilon}$ is defined as in (\ref{eq:iota from distribs to gen func on X}). This is established by the same arguments as in the proof of \citep[Lem.~3.2]{hawkingc11}:
Clearly the only problematic terms are the ones involving second derivatives
of the metric (all other terms converge to the respective ones of
$\mathbf{Ric}$ in $\mathcal{C}^{0}$). Now on every compact
set $g_{\varepsilon}=\tilde{g}_{\varepsilon}$ for $\varepsilon$
small enough by construction, so the terms involving second
derivatives of $g$ are dealt with using a variant of the Friedrichs
lemma, showing that for any $f\in\mathcal{C}^{0}(\mathbb{R}^{n})$
and $g\in L_{\mathrm{loc}}^{\infty}$ the difference $f_{\varepsilon}(h*\rho_{\varepsilon})-(f\, h)*\rho_{\varepsilon}\to0$
if $f_{\varepsilon}\to f$ in $\mathcal{C}^{0}$ (cf. \citep[Lem.~3.2]{hawkingc11}).

Now let $\delta>0$ and $K\subset M$ compact (and w.l.o.g.\ contained
in some chart domain). If we define $A_{\varepsilon}:=\mathbf{Ric}_{\varepsilon}-\left(n-1\right)\kappa g_{\varepsilon}$
and $A:=\mathbf{Ric}-\left(n-1\right)\kappa g$, then clearly also
$A_{\varepsilon}-\tilde{A}_{\varepsilon}\to0$ uniformly on $K$. So for any $X\in TM|_{K}$
\[
\left|A_{\varepsilon}(X,X)-\tilde{A}_{\varepsilon}(X,X)\right|\leq n\, C^{2}|X|_{g}^{2}\,\sup_{i,j\leq n}\sup_{p\in K}\left|A_{\varepsilon,ij}(p)-\tilde{A}_{\varepsilon,ij}(p)\right|\leq\delta\,(n-1)\, g(X,X)
\]
for $\varepsilon$ small (this follows by similar estimates as in
(\ref{eq:blubb})). So if we can show that $\tilde{A}_{\varepsilon}(X,X)\geq0$
for all $X\in TM|_{K}$ the claim follows. By construction $\tilde{A}_{\varepsilon}|_{K}$
is a finite sum of terms of the form $\zeta_{\alpha}\psi_{\alpha}^{*}((\tilde{\chi}_{\alpha}A_{ij})*\rho_{\varepsilon})$
(see (\ref{eq:iota from distribs to gen func on X})) so it suffices
to show that $\left((\tilde{\chi}_{\alpha}A_{ij})*\rho_{\varepsilon}\right)(p)$
is a positive semi-definite matrix for any $p\in\psi_{\alpha}(\mathrm{supp}\zeta_{\alpha})$
(note that $(\tilde{\chi}_{\alpha}A_{ij})*\rho_{\varepsilon}$ is
well defined on an open neighborhood $U$ of $\psi_{\alpha}(\mathrm{supp}\zeta_{\alpha})$
contained in $\psi_{\alpha}(U_{\alpha})$ for $\varepsilon$ small
enough). Now let $p\in\psi_{\alpha}(\mathrm{supp}\zeta_{\alpha})$
and $X_{p}\in\mathbb{R}^{n}$ and let $\tilde{X}$ be the constant
vector field $x\mapsto X_{p}$ on $\psi_{\alpha}(U_{\alpha})$. Then
\[
\left((\tilde{\chi}_{\alpha}A_{ij})*\rho_{\varepsilon}\right)(p)X_{p}^{i}X_{p}^{j}=((\tilde{\chi}_{\alpha}A_{ij}X_{p}^{i}X_{p}^{j})*\rho_{\varepsilon})(p)\geq0
\]
since $x\mapsto\tilde{\chi}_{\alpha}(x)A_{ij}(x)\tilde{X}^{i}(x)\tilde{X}^{j}(x)=\tilde{\chi}_{\alpha}(x)A_{ij}(x)X_{p}^{i}X_{p}^{j}$
is non-negative in $L_{\mathrm{loc}}^{\infty}$ by assumption and $\rho_{\varepsilon}\geq0$.
\end{proof}
These preparations now enable us to show:
\begin{theorem}
[Volume comparison]\label{thm:grmov volume}Let $\left(M,g\right)$
be a complete Riemannian manifold with $g\in\mathcal{C}^{1,1}$ and
$\mathbf{Ric}\geq\left(n-1\right)\kappa\, g$. Then
\[
r\mapsto\frac{\mathrm{vol}B_{p}(r)}{\mathrm{vol}_{\kappa}B^{\kappa}(r)}
\]
is a nonincreasing function on $(0,\infty)$ and $\mathrm{vol}B_{p}(r)\leq\mathrm{vol}_{\kappa}B^{\kappa}(r)$.\end{theorem}
\begin{proof}
Let $p\in M$ and $0<r_{1}<r_{2}<R$. Using the approximating metrics
$g_{\varepsilon}$ constructed in Lem.\ \ref{lem:riem approx} and
\ref{lem:Riem approx ricci} we see that for any $\delta>0$ there
exists some $\varepsilon_{0}$ such that $\left(B_{p}(R),g_{\varepsilon}\right)$
(as a submanifold of $M$) satisfies the conditions of the classical
Bishop-Gromov volume comparison (Thm.\ \ref{thm:classical bishop gromov})
with $\mathbf{Ric}_{\varepsilon}\geq\left(n-1\right)\left(\kappa-\delta\right)g_{\varepsilon}$
for all $\varepsilon\leq\varepsilon_{0}$. This gives us
\[
1\geq\frac{\mathrm{vol}_{\varepsilon}B_{p}(r_{1})}{\mathrm{vol}_{\kappa-\delta}B^{\kappa-\delta}(r_{1})}\geq\frac{\mathrm{vol}_{\varepsilon}B_{p}(r_{2})}{\mathrm{vol}_{\kappa-\delta}B^{\kappa-\delta}(r_{2})}.
\]
Now by (\ref{eq: distance estimates}) from the proof of Lem.\ \ref{lem:riem approx}
it follows that $d_{g_{\varepsilon}}(p,q)\to d_{g}(p,q)$ and hence
for any $r>0$ one has that $\chi_{B_{\varepsilon,p}(r)}\to\chi_{B_{p}(r)}$
almost everywhere (because the sphere $S_{p}(r)\subset\exp_{p}(r\cdot\{v\in T_{p}M:\,\left|v\right|_{g}=1\})$,
which has measure zero since $\exp_{p}$ is still locally Lipschitz
and $r\cdot \{v\in T_{p}M:\,\left|v\right|_{g}=1\}\subset T_{p}M$
has measure zero). So by dominated convergence (note that $B_{\varepsilon,p}(r)\subset B_{p}(\frac{r}{\sqrt{1-\delta}})$
by (\ref{eq: distance estimates}) for $\varepsilon$ small, hence
the support of all characteristic functions is contained in a common
compact set) $\mathrm{vol}_{\varepsilon}B_{p}(r)\to\mathrm{vol}B_{p}(r)$
for all $r>0$. Calculating the volumes of balls in the comparison
spaces shows that $\mathrm{vol}_{\kappa-\delta}B^{\kappa-\delta}(r)=c\,\int_{0}^{r}\mathrm{sn}_{\kappa-\delta}(s)^{n-1}ds\to c\,\int_{0}^{r}\mathrm{sn}_{\kappa}(s)^{n-1}ds=\mathrm{vol}_{\kappa}B^{\kappa}(r)$,
where
\[
\mathrm{sn}_{\kappa}(s):=\begin{cases}
\frac{1}{\sqrt{\kappa}}\sin(\sqrt{\kappa}s) & \kappa>0\\
s & \kappa=0\\
\frac{1}{\sqrt{|\kappa|}}\sinh(\sqrt{|\kappa|}s) & \kappa<0,
\end{cases}
\]
for $\delta\to0$. Altogether this proves the theorem.
\end{proof}

\section{\label{sec:The-Lorentzian-case}The Lorentzian case}

In this section the goal is to use volume comparison results (as developed in \citep{TG}) for smooth,
globally hyperbolic spacetimes $M$ with timelike Ricci curvature
bounded from below and containing a spacelike hypersurface $\Sigma$ (satisfying
some additional causality and completeness conditions) that has mean
curvature bounded from above to establish
analogous results for $\mathcal{C}^{1,1}$-metrics. It should be noted
that these conditions are very similar to those of the Hawking singularity
theorem and \citep{TG} includes proofs of this theorem using the
new comparison techniques therein. So one of the motivations of this
paper was to also give an alternative proof of Hawking's singularity
theorem in $\mathcal{C}^{1,1}$-regularity (which was first shown
in \citep{hawkingc11}). This will be done in section \ref{sub:Hawking's-singularity-theorem}.

However, there are some additional difficulties (compared to the Riemannian
result from the previous section) arising due to the metric being Lorentzian: First, one
has to be more careful when choosing approximating metrics and simple
convolution is no longer sufficient since it need not preserve the
causal structure. Here the pioneering work was done by Chru\'{s}ciel
and Grant in \citep{CG}, and from there on causality theory for $\mathcal{C}^{1,1}$
metrics has been developed (see, e.g., \citep{Minguzzi_convexnbhdsLipConSprays,KSSV,hawkingc11}).
Additionally, the concept of global hyperbolicity for continuous metrics
has recently been explored in \citep{clemensGlobHYp}. This will
be helpful in establishing certain results from causality theory for
globally hyperbolic spacetimes with a $\mathcal{C}^{1,1}$-metric
in section \ref{sub:Basic-definitions-and}.

Second, while there is no assumption of (geodesic) completeness needed
for the smooth result, an assumption on the minimal time of existence
of geodesics starting orthogonally to the hypersurface with unit
speed has to be made to ensure that everything plays out in relatively
compact sets.

Third, showing that the volumes of the balls in the approximating
metrics actually converge to the volumes in the $\mathcal{C}^{1,1}$-metric
is a bit more involved and will need a result
regarding the cut locus of $\Sigma$ with respect to the $\mathcal{C}^{1,1}$-metric,
namely that it has measure zero. This will be shown in section \ref{sub:The-cut-locus}.

\subsection{\label{sub:Basic-definitions-and}Basic definitions and results}

Throughout this section $M$ will always be a Lorentzian manifold
with a time orientation. While we will generally assume $\mathcal{C}^{1,1}$
 regularity of the metric, we will often include this assumption
explicitly to highlight its importance (many of our results will be
both well-known in higher and not true, or at least unproven, in lower
regularity). We also fix once and for all a (complete) Riemannian
background metric $h$ on $M$.  

As in, e.g., \citep{BEE96,Chrusciel_causality} we define causal
(timelike) curves to be locally Lipschitz continuous maps $\gamma:I\to M$
($I$ being an interval) with $\dot{\gamma}\neq0$ and $g(\dot{\gamma},\dot{\gamma})\leq0$
($<0$) almost everywhere. A causal curve is called future (past) directed
if $\dot{\gamma}$ is future (past) pointing almost everywhere.

For $p,q\in M$ we write $p\ll q$ if there exists a future directed (f.d.)
timelike curve from $p$ to $q$ and $p\leq q$ if either $p=q$ or
there exists a f.d.\ causal curve from $p$ to $q$. We also define
\begin{align*}
I^{+}(p): & =\left\{ q\in M:\, p\ll q\right\} \\
J^{+}(p): & =\{q\in M:\, p\leq q\}.
\end{align*}
$I^{-}$ and $J^{-}$ are defined analogously. Note that for a $\mathcal{C}^{1,1}$-metric
it does not matter whether one allows Lipschitz causal curves or one
requires causal curves to be piecewise $\mathcal{C}^{1}$ (or even
broken geodesics) in the definition of $I^{+}$ and $J^{+}$ (see
\citep[Thm. 1.27]{Minguzzi_convexnbhdsLipConSprays} or \citep[Cor. 3.10]{KSSV}).
Note also that most results from smooth causality theory carry over
to $\mathcal{C}^{1,1}$-metrics, we refer to \citep{Minguzzi_convexnbhdsLipConSprays,KSSV}
and \citep[Appendix A]{hawkingc11} for an overview.

We will mainly work with globally hyperbolic manifolds and as for
smooth metrics one may use any of the following equivalent properties
as definition.
\begin{prop}
[Global hyperbolicity]\label{prop:glob hyp def+equivalences}Let
$\left(M,g\right)$ be a spacetime with $\mathcal{C}^{1,1}$-metric
$g$. Then the following properties are equivalent:
\begin{enumerate}
\item $\left(M,g\right)$ is causal and for all $p,q\in M$ the set $J(p,q):=J^{+}(p)\cap J^{-}(q)$
is compact,
\item there exists a Cauchy hypersurface $S$ for $M$ (i.e.\ a set $S\subset M$
that is met exactly once by every inextendible timelike curve) and
\item $\left(M,g\right)$ is causal and $C(p,q)$ (the space of equivalence
classes of future directed causal curves from $p$ to $q$ with the
compact-open topology) is compact
\end{enumerate}
If any of these conditions holds, we say that $\left(M,g\right)$
is globally hyperbolic.\end{prop}
\begin{proof}
In \citep{clemensGlobHYp} it was shown that these are equivalent
even for continuous metrics, if one replaces causality with the slightly
stronger assumption of $\left(M,g\right)$ being non-totally imprisoning.
So it only remains to show that for a $\mathcal{C}^{1,1}$-metric
both (1) and (3) already imply $M$ being non-totally imprisoning.
This follows as for smooth metrics so we will only present a brief
outline: From compactness of $J(p,q)$ (respectively $C(p,q)$) one
obtains that $J^{\pm}(p)$ is closed for all $p$, see \citep[Prop.~3.71]{Minguzzi08thecausal},
respectively \citep[Prop.~3.3]{clemensGlobHYp} (note that the proof
only actually uses compactness of $C(p,q)$). Since $g\in\mathcal{C}^{1,1}$
one can still use the exponential map to show that then already $J^{\pm}(p)=\overline{I^{\pm}(p)}$
(\citep[Cor.~3.16]{KSSV}). Thus $\left(M,g\right)$ is distinguishing
and reflective (\citep[Prop.~3.64 and 3.65]{Minguzzi08thecausal}),
hence strongly causal (Prop.\ 3.41, 3.47 and Thm.\ 3.51 in \citep{Minguzzi08thecausal}
show the existence of a time function and Prop.\ 3.57 gives strong
causality). That strong causality is stronger than non-totally imprisoning follows
again as in the smooth case (see e.g.\ \citep[Lem.~14.13]{ONeill_SRG})
as was already remarked in \citep{PenroseC11inpaper}.\end{proof}
\begin{rem}
The previous proof also shows that for $\mathcal{C}^{1,1}$-metrics
this definition of global hyperbolicity is equivalent to the one in
\citep{clemensGlobHYp}. \end{rem}
\begin{defn}
[Future time separation] Let $p\in M$. Then for $q\in M$ the (future)
time separation to $p$ is defined by 
\begin{equation}
\tau(p,q):=\sup(\left\{ L(\gamma):\gamma\,\text{is a f.d. causal curve form }p\text{ to }q\right\} \cup\{0\}),\label{eq:point time sep}
\end{equation}
where $L(\gamma)$ denotes the Lorentzian arc-length of $\gamma$, i.e., for a curve $\gamma:(t_1,t_2)\to M$ one has $L(\gamma):=\int_{t_1}^{t_2}\sqrt{|g(\dot{\gamma(t)},\dot{\gamma(t)})|}dt$.  Similarly one defines the future time separation to a subset $\Sigma$
by
\begin{equation}
\tau_{\Sigma}(p):=\sup_{q\in\Sigma}\tau(q,p).\label{eq:subset time sep}
\end{equation}

\end{defn}
If $M$ is globally hyperbolic with a continuous metric then any two
causally related points can be connected by a maximizing curve (\citep[Prop.~6.4]{clemensGlobHYp}),
hence the supremum in definition (\ref{eq:point time sep}) is attained,
so $\tau:M\times M\to\left[0,\infty\right]$ is finite-valued. It
is also lower semi-continuous (this holds even if $M$ is not globally
hyperbolic, see \citep[Lem.~A.16]{hawkingc11}). We want to show a
similar statement for the time separation to a subset $\Sigma$. This
requires some additional properties of $\Sigma$ (\citep[Def.~2]{TG}). 
\begin{defn}
[Future causally complete] A subset $\Sigma\subset M$ is called
\emph{future causally complete} (FCC) if for any $p\in J^{+}(\Sigma)$
the set $J^{-}(p)\cap\Sigma$ has compact relative closure in $\Sigma$.\end{defn}
\begin{rem}
\label{rem: J-capSig compact}In a globally hyperbolic manifold the
sets $J^{\pm}(p)$ are closed (\citep[Prop. 3.3]{clemensGlobHYp})
and hence for any FCC subset $\Sigma$ and $p\in J^{+}(\Sigma)$ we
have that $J^{-}(p)\cap\Sigma$ is compact and $\Sigma$ itself is
closed. Furthermore, from \citep[Cor.~3.4]{clemensGlobHYp}, it then
follows that 
\begin{equation}
J^{-}(p)\cap J^{+}(J^{-}(p)\cap\Sigma)\quad\mathrm{is\: compact.}\label{eq:J-capSig compact}
\end{equation}

\end{rem}
As a preparation for Prop.\ \ref{prop: time sep continuous} we prove
the following limit-curve lemma (that will also be needed again later
on), which is a slight modification of Thm.\ 1.5 in \citep{clemensGlobHYp}
(which is in turn based on \citep{Minguzzi_LimitCurveThms}):
\begin{lem}
\label{lem:Limit curve}Let $M$ be globally hyperbolic and $\gamma_{n}:\left[0,1\right]\to M$
be a sequence of causal curves and $K\subset M$ compact such that
$\gamma_{n}\subset K$ for all $n\in\mathbb{N}$. Then there exists
a subsequence $\gamma_{n_{k}}$ that converges ($h$-)uniformly to
a causal curve $\gamma:\left[0,1\right]\to M$ (i.e.\ $\sup_{t\in\left[0,1\right]}d_{h}(\gamma_{n_{k}}(t),\gamma(t))\to0$)
with
\begin{equation}
L(\gamma)\geq\limsup_{k\to\infty}L(\gamma_{n_{k}}).\label{eq:L usc}
\end{equation}
In particular, if the $\gamma_{n}$ are maximizing, then $\gamma$
is as well.\end{lem}
\begin{proof}
By \citep[Lem.~2.7]{clemensGlobHYp} we get an upper bound on the
Lipschitz constants of the $\gamma_{n}$. And so, since the sequence
must have an accumulation point, the convergence result follows from
Thm.\ 1.5 of \citep{clemensGlobHYp}.

It remains to show (\ref{eq:L usc}) and that $\gamma$ is maximizing
if the $\gamma_{n}$ are. By \citep[Thm.~6.3]{clemensGlobHYp} the
length functional $L:\left\{ \gamma\in\mathcal{C}\left(\left[0,1\right],K\right):\,\gamma\,\text{causal}\right\} \to[0,\infty)$
is upper semi-continuous w.r.t.\ $h$-uniform convergence as defined
above (note that while the statement there only deals with a special
subset of causal curves defined on $\left[0,1\right]$, the proof
works for any set of such curves with an upper bound on the Lipschitz
constants), so $L(\gamma)\geq\limsup L(\gamma_{n_{k}})$. Using this
and lower semi-continuity of $\tau$ (see \citep[Lem.~A.16]{hawkingc11})
gives 
\[
L(\gamma)\geq\limsup L(\gamma_{n_{k}})=\limsup\tau(\gamma_{n_{k}}(0),\gamma_{n_{k}}(1))\geq\tau\left(\gamma(0),\gamma(1)\right),
\]
so $\gamma$ is maximizing.
\end{proof}
For an acausal, spacelike FCC hypersurface in a globally hyperbolic
manifold the following holds (which is shown largely analogous to
the smooth case (\citep[Thm.~2]{TG}), only using Lem.\ \ref{lem:Limit curve}
instead of other limit curve results, we nevertheless include a complete
proof):
\begin{prop}
\label{prop: time sep continuous}Let $\left(M,g\right)$ with $g\in\mathcal{C}^{1,1}$
be globally hyperbolic and let $\Sigma\subset M$ be an acausal, FCC subset.
Then the future time-separation $\tau_{\Sigma}$
to $\Sigma$ is finite-valued and continuous on $M$ and for any $p\in J^{+}(\Sigma)\setminus\Sigma$
there exists $q\in\Sigma$ and a causal curve $\gamma$ from $q$
to $p$ with $\tau_{\Sigma}(p)=\tau(q,p)=L(\gamma)$. Any such maximizing
curve $\gamma$ has to be a (reparametrization of) a geodesic, which
is timelike for $p\in I^{+}(\Sigma)$ and null otherwise. If $\Sigma\subset M$
is, additionally, a spacelike hypersurface, then for $p\in I^{+}(\Sigma)$
any maximizing geodesic has to start orthogonally to $\Sigma$.\end{prop}
\begin{proof}
If $p\notin I^{+}(\Sigma)$ then $\tau_{\Sigma}(p)=0$. Now let $p\in J^{+}(\Sigma)\setminus\Sigma$.
Then there exists a causal curve $\gamma$ from $p$ to $q\in\Sigma$
and if $p\notin I^{+}(\Sigma)$ then clearly $L(\gamma)\leq\tau_{\Sigma}(p)=0\leq L(\gamma)$.
So assume $p\in I^{+}(\Sigma)$. By definition of $\tau_{\Sigma}$
there exist $q_{n}\in\Sigma$ such that $\tau(q_{n},p)\to\tau_{\Sigma}(p)$.
Since $p\in I^{+}(\Sigma)$ we have $\tau_{\Sigma}(p)>0$ and hence
$\tau(q_{n},p)>0$ for $n$ large, so $q_{n}$ and $p$ are causally
related and can be connected by a maximizing curve $\gamma_{n}$ (see
\citep[Prop.~6.4]{clemensGlobHYp}). Because $q_{n}\in J^{-}(p)\cap\Sigma$,
all the $\gamma_{n}$ are contained in $J^{-}(p)\cap J^{+}(J^{-}(p)\cap\Sigma)$,
which is compact by Rem.\ \ref{rem: J-capSig compact}. Therefore
(after maybe reparametrizing and passing to a subsequence), Lem.\ \ref{lem:Limit curve}
gives a uniform limit curve $\gamma$ that is causal, satisfies $q=\gamma(0)\in\Sigma$
(note that $\Sigma$ is closed by Rem.\ \ref{rem: J-capSig compact})
and $p=\gamma(1)$ and is maximizing, so by upper semi-continuity
of the length functional we get
\[
\tau(p,q)=L(\gamma)\geq\limsup L(\gamma_{n})=\limsup\tau(q_{n},p)=\tau_{\Sigma}(p).
\]
Consequently, $\gamma$ maximizes the distance from $\Sigma$ to $p$
and $\tau_{\Sigma}(p)$ is finite. 

Regarding continuity we show lower and upper semi-continuity separately,
starting with lower semi-continuity. Let $p\in M$. We have to show
that for every $\varepsilon$ there exists a neighborhood $U_{\varepsilon}$
of $p$ such that for all $q\in U_{\varepsilon}$
\[
\tau_{\Sigma}(q)\geq\tau_{\Sigma}(p)-\varepsilon.
\]
If $\tau_{\Sigma}(p)=0$, there is nothing to prove due to non-negativity
of $\tau_{\Sigma}$. Let $\gamma:\left[0,1\right]\to M$ be a causal
curve from $p_{0}\in\Sigma$ to $p$ such that $L(\gamma)=\tau(p_{\text{0}},p)=\tau_{\Sigma}(p)>0$.
Now for any $\varepsilon>0$ there exists $t_{\varepsilon}$ such
that $L(\gamma|_{\left[t_{\varepsilon},1\right]})<\varepsilon$. Then
$U_{\varepsilon}:=I^{+}(\gamma(t_{\varepsilon}))$ is a neighborhood
of $p$ such that for all $q\in U_{\varepsilon}$
\[
\tau_{\Sigma}(q)\geq L(\gamma|_{\left[0,t_{\varepsilon}\right]})=\tau_{\Sigma}(p)-L(\gamma|_{\left[t_{\varepsilon},1\right]})\geq\tau_{\Sigma}(p)-\varepsilon.
\]
Next we show upper semi-continuity, i.e., for every $\varepsilon$
there exists a neighborhood $U_{\varepsilon}$ of $p$ such that for
all $q\in U_{\varepsilon}$
\[
\tau_{\Sigma}(q)\leq\tau_{\Sigma}(p)+\varepsilon.
\]
Assume to the contrary that there exists $\varepsilon>0$ and $p_{n}\to p$
such that
\[
\tau_{\Sigma}(p_{n})>\tau_{\Sigma}(p)+\varepsilon
\]
and let $\gamma_{p_{n}}:\left[0,1\right]\to M$ be causal curves from
$\Sigma$ to $p_{n}$ with $\tau_{\Sigma}(p_{n})=L(\gamma_{p_{n}})$
(such curves exist, since $\tau_{\Sigma}(p_{n})>\tau_{\Sigma}(p)+\varepsilon>0$
and so $p_{n}\in I^{+}(\Sigma)$). Let $p^{+}\in I^{+}(p)$, then
$p_{n}\in J^{-}(p^{+})$ eventually and thus $\gamma_{p_{n}}\subset J^{-}(p^{+})\cap J^{+}(J^{-}(p^{+})\cap\Sigma)$,
which is compact by Rem.\ \ref{rem: J-capSig compact}. So we can
apply Lem.\ \ref{lem:Limit curve} to obtain (after passing to a
subsequence) a curve $\gamma$ from $\Sigma$ to $p=\lim p_{n}$ with
\[
\tau_{\Sigma}(p)\geq L(\gamma)\geq\limsup_{n\to\infty}L(\gamma_{p_{n}})=\limsup_{n\to\infty}\tau_{\Sigma}(p_{n})\geq\tau_{\Sigma}(p)+\varepsilon
\]
which is a contradiction. 

Since causal geodesics are locally maximizing (by \citep[Thm.~6]{Minguzzi_convexnbhdsLipConSprays}),
any maximizing curve must be (a reparametrization of) a geodesic and
if $p\in I^{+}(\Sigma)$ then $\tau_{\Sigma}(p)>0$, so it has to
be timelike. 

Now let $\Sigma$ be an acausal, FCC, spacelike hypersurface. We show
that all timelike geodesics that start in $\Sigma$ and maximize the
distance to $\Sigma$ must start orthogonally: First note that if
$\gamma:\left[0,b\right]\to M$ maximizes the distance then also $\gamma|_{\left[0,\varepsilon\right]}$
must maximize the distance to $\Sigma$, so this is a local question
and we may assume that $M=\mathbb{R}^{n}$, $\Sigma\subset\mathbb{R}^{n}$
is a hypersurface and $\gamma:\left[0,1\right]\to\mathbb{R}^{n}$
is a timelike unit-speed geodesic with $\gamma(0)=0\in\Sigma$ that
maximizes the distance to $\Sigma$. Now for any $v\in T_{0}\Sigma$
we can find a smooth curve $\alpha:\left[0,\varepsilon\right]\to\Sigma$
such that $\dot{\alpha}(0)=v$ and $\alpha(0)=0$. We use this to
define a $\mathcal{C}^{2,1}$ (note that $\gamma$ is a geodesic,
hence $\mathcal{C}^{2,1}$ by the geodesic equation) variation
\begin{align*}
\sigma:\left[0,1\right]\times\left[0,\varepsilon\right] & \to\mathbb{R}^{n}\\
\sigma(t,s) & =\gamma(t)+(1-t)\alpha(s).
\end{align*}
Since $\gamma$ is timelike this is a timelike variation for small
enough $\varepsilon$ and we may use the first variation of arc-length
(see \citep[Prop.~10.2]{ONeill_SRG} and note that $s\mapsto L(\sigma(.,s))$
is still $\mathcal{C}^{1}$) to obtain
\[
0=L'(0)=\int_{0}^{1}g(\ddot{\gamma}(t),(\partial_{s}\sigma)(t,0))dt+g(v,\dot{\gamma}(0))=g(v,\dot{\gamma}(0)).
\]
This shows that $\dot{\gamma}(0)\perp v$ for all $v\in T_{0}\Sigma$,
so $\gamma$ starts orthogonally.
\end{proof}
Note that the part of the proof that shows that $\gamma$ has to start
orthogonally to $\Sigma$ really only works for $p\in I^{+}(\Sigma)$
and not for $p\in J^{+}(\Sigma)$ since in that case one could not
guarantee that the constructed variation consists only of timelike
curves. However, the next remark shows that $J^{+}(\Sigma)\setminus\left(\Sigma\cup I^{+}(\Sigma)\right)=\emptyset$
anyways.
\begin{rem}
\label{rem: FCC things}If $\Sigma$ is an acausal, FCC hypersurface
then actually $J^{+}(\Sigma)\setminus\Sigma=I^{+}(\Sigma)$. The argument
is the same as for smooth metrics: First, any FCC set must be closed
(by Rem.\ \ref{rem: J-capSig compact}) and then \citep[Cor.~14.26]{ONeill_SRG}
shows that $\mathrm{edge}(\Sigma)=\emptyset$. By Prop.\ \ref{prop: time sep continuous}
any $p\in J^{+}(\Sigma)\setminus\left(\Sigma\cup I^{+}(\Sigma)\right)$
is the future endpoint of a lightlike geodesic $\gamma$ starting
in $\gamma(0)\in\Sigma$. Now if $\gamma(0)\notin\mathrm{edge}(\Sigma)$
then for $\varepsilon$ small enough $\gamma(\varepsilon)\in I^{+}(\Sigma)$
(since by definition of $\mathrm{edge}(\Sigma)$ (\citep[Def.~14.23]{ONeill_SRG})
there must exist a $q^{-}\in I^{-}(\gamma(0))$ such that any timelike
curve connecting $q^{-}$ to $\gamma(\varepsilon)$ meets $\Sigma$)
contradicting $p\notin I^{+}(\Sigma)$. But since $\mathrm{edge}(\Sigma)=\emptyset$
this shows that $J^{+}(\Sigma)\setminus\left(\Sigma\cup I^{+}(\Sigma)\right)=\emptyset$.
\end{rem}
Finally, we will specify the curvature conditions, introduced in
\citep[Def.~5]{TG}, $\left(M,g\right)$ has to satisfy for the volume
comparison theorem (Thm.\ \ref{thm: volume}) we are going to show.
\begin{defn}
[Cosmological comparison condition] \label{def: CCC}Let $\kappa,\beta\in\mathbb{R}$.
We say that a spacetime $\left(M,g,\Sigma\right)$ satisfies the cosmological
comparison condition $CCC(\kappa,\beta)$ if
\begin{enumerate}
\item $\Sigma\subset M$ is a smooth, spacelike, acausal, FCC hypersurface
and the mean curvature $H$ of $\Sigma$ satisfies $H\leq\beta$ and
\item $\mathbf{Ric}(X,X)\geq-\left(n-1\right)\kappa\, g(X,X)$ in
$L_{\mathrm{loc}}^{\infty}$ for any local, smooth timelike vector field $X$ (i.e., the timelike Ricci curvature is bounded from below by $\kappa$ in the sense of (\ref{eq:ric est}))
\end{enumerate}
\end{defn}
\begin{rem}
\label{rem: conventions for shape op and mean curv}Following \citep{ONeill_SRG}
our sign conventions regarding the mean curvature are as follows:
Let $\Sigma$ be a spacelike hypersurface and $\mathbf{n}$ be the
f.d.\ timelike unit normal vector field to $\Sigma$. We define the
shape operator $S_{\mathbf{n}}:T\Sigma\to T\Sigma$ of $\Sigma$ by
$S_{\mathbf{n}}(V):=\mathrm{tan}_{\Sigma}\nabla_{V}\mathbf{n}$, where
$\mathrm{tan}_{\Sigma}$ denotes the tangential projection $TM|_{\Sigma}\to T\Sigma$.
Using the shape operator we can write the mean curvature as $H=\mathrm{tr}_{g|_{T\Sigma}}\, S_{\mathbf{n}}$,
where $g|_{T\Sigma}$ denotes the metric on $\Sigma$ induced by $g$.
\end{rem}
Note that even though basically all of the upcoming results (except
for the $\mathcal{C}^{1,1}$ version of Hawking's theorem at the very
end, see Thm.\ \ref{thm:C11 hawking for M not glob hyp}) will additionally
require global hyperbolicity, we choose not to include this in the
definition of the comparison condition.

\subsection{\label{sub:Desired-properties-of}Construction and properties of
the approximating metrics}

We need to establish some properties of suitable approximations for
a $\mathcal{C}^{1,1}$-metric $g$ with a hypersurface $\Sigma$ satisfying
$CCC(\kappa,\beta)$. This is done in the following three Lemmas.
To start with, we use approximations as constructed in \citep{CG},
i.e., we have (using the formulation of \citep[Prop.~2.5]{KSSV}):
\begin{prop}
\label{prop:basic props of approx}Let $\left(M,g\right)$ be a space-time
with a continuous Lorentzian metric, and $h$ some smooth background
Riemannian metric on M. Then for any $\varepsilon>0$, there exist
smooth Lorentzian metrics $\check{g}_{\varepsilon}$ and $\hat{g}_{\varepsilon}$
on $M$ such that $\check{g}_{\varepsilon}\prec g\prec\hat{g}_{\varepsilon}$,
i.e.\
\[
\forall X\in TM:\:\check{g}_{\varepsilon}(X,X)\leq0\implies g(X,X)<0\:\mathrm{and}\: g(X,X)\leq0\implies\hat{g}_{\varepsilon}(X,X)<0,
\]
and $d_{h}(\check{g}_{\varepsilon},g)+d_{h}(\hat{g}_{\varepsilon},g)<\varepsilon$,
where
\[
d_{h}(g_{1},g_{2}):=\sup_{p\in M}\sup_{0\neq X,Y\in T_{p}M}\frac{\left|g_{1}(X,Y)-g_{2}(X,Y)\right|}{\left\Vert X\right\Vert _{h}\left\Vert Y\right\Vert _{h}}.
\]
Moreover, $\check{g}_{\varepsilon}$ and $\hat{g}_{\varepsilon}$
depend smoothly on $\varepsilon$, and if $g\in\mathcal{C}^{1,1}$
then $\check{g}_{\varepsilon}$ and $\hat{g}_{\varepsilon}$ additionally
satisfy
\begin{enumerate}
\item they converge to $g$ in the $\mathcal{C}^{1}$-topology as $\varepsilon\to0$
and
\item the second derivatives are bounded, uniformly in $\varepsilon$, on
compact sets.
\end{enumerate}
\end{prop}
Now we show that we may additionally demand the following:
\begin{lem}
\label{lem: FCC}Let $(M,g)$ be globally hyperbolic with $g\in\mathcal{C}^{1,1}$
and let $\Sigma\subset M$ be a smooth acausal, spacelike FCC hypersurface.
Then there exist smooth approximations $g_{\varepsilon}$ such that
the approximations $\left(M,g_{\varepsilon}\right)$ are globally
hyperbolic and $\Sigma\subset M$ is a smooth acausal, spacelike FCC
hypersurface (w.r.t.\ $g_{\varepsilon}$).\end{lem}
\begin{proof}
We first show that we can construct approximations $g_{\varepsilon}$
that retain the properties of the $\check{g}_{\varepsilon}$ from
above but additionally satisfy that $\Sigma$ is $g_{\varepsilon}$
spacelike for $\varepsilon$ small, i.e.\ $g_{\varepsilon}|_{T\Sigma}$
is positive definite. To do this, we show that for every compact set
$K\subset\Sigma$ there exists $\varepsilon_{K}$ such that this holds
for the $\check{g}_{\varepsilon}$ for all $\varepsilon\leq\varepsilon_{K}$
and then apply the globalization lemma (\citep[Lem.~2.4]{KSSV}).
 This gives us metrics $g_{\varepsilon}(p):=\check{g}_{\tilde{\varepsilon}(\varepsilon,p)}(p)$
that satisfy $g_{\varepsilon}|_{K}=\check{g}_{\varepsilon}|_{K}$
for all $\varepsilon\leq\varepsilon_{K}$ and $g_{\varepsilon}|_{T\Sigma}$
is positive definite. Since $g|_{T\Sigma}$ is a Riemannian metric
on $\Sigma$ we have that $g(X,X)=1$ implies $\left\Vert X\right\Vert _{h}\leq C$
for all $X\in T\Sigma|_{K}$ and hence $\sup_{\{X\in T\Sigma|_{K}:g(X,X)=1\}}\check{g}_{\varepsilon}(X,X)-g(X,X)\to0$
by the previous proposition. So $\check{g}_{\varepsilon}(X,X)>c\, g(X,X)>0$
for any nonzero $X\in T\Sigma|_{K}$ for all $\varepsilon$ small
(depending on $K$), showing positive definiteness.

The other properties follow because by the above construction $g_{\varepsilon}\prec g$
(since $g_{\varepsilon}(p)=\check{g}_{\tilde{\varepsilon}(\varepsilon,p)}(p)$
and $\check{g}_{\varepsilon}\prec g$): By Prop.\ \ref{prop:glob hyp def+equivalences},
global hyperbolicity is equivalent to the existence of a Cauchy hypersurface
and by definition any Cauchy hypersurface for $g$ also has to be
a Cauchy hypersurface for any $g'\prec g$. This shows that $\left(M,g_{\varepsilon}\right)$
is globally hyperbolic. Similarly $\Sigma$ being $g$-FCC implies
$g_{\varepsilon}$-FCC and $g$-acausality of $\Sigma$ implies $g_{\varepsilon}$-acausality.
\end{proof}
From now on $g_{\varepsilon}$ will always denote smooth approximating
metrics as constructed above, in particular satisfying Prop.\ \ref{prop:basic props of approx},
Lem.\ \ref{lem: FCC} and $g_{\varepsilon}\prec g$. The next Lemma
shows properties of the Ricci curvature $\mathbf{Ric}_{\varepsilon}$
of this approximations (which is basically \citep[Lem.~3.2]{hawkingc11},
except also explicitly covering the case $\kappa\neq0$, and the proof
proceeds similarly).
\begin{lem}
\label{lem: ricci}Let $g\in\mathcal{C}^{1,1}$ and $h$ be a background
Riemannian metric. Suppose that $\mathbf{Ric}_{g}(X,X)\geq-n\,\kappa\, g(X,X)$
for any local smooth $g$-timelike vector field $X\in\mathfrak{X}(U)$.
Then for any compact set $K\subset M$, $C>0$ and $\delta>0$ there
exists $\varepsilon_{0}=\varepsilon_{0}(K,C,\delta)$ such that
\begin{equation}
\mathbf{Ric}_{\varepsilon}(X,X)\geq(n-1)\,(\kappa-\delta)\quad\forall X\in TM|_{K}\:\mathrm{with}\: g_{\varepsilon}(X,X)=-1\:\mathrm{and}\:\left\Vert X\right\Vert _{h}<C\label{eq:ricci for eps metric}
\end{equation}
 for all $\varepsilon<\varepsilon_{0}$.\end{lem}
\begin{proof}
Fix $K\subset M$ (w.l.o.g.\ contained in a chart domain), $C>0$
and $\delta>0$. As in the proof of Lem.\ \ref{lem:Riem approx ricci}
we proceed similarly to \citep[Lem.~3.2]{hawkingc11}. By the argument
given there $g_{\varepsilon}-\tilde{g}_{\varepsilon}\to0$ in $\mathcal{C}^{2}$
(note that by construction $g_{\varepsilon}=\check{g}_{\varepsilon}$
on $K$ for $\varepsilon$ small). As in (\ref{eq:R-R*rho_eps goes to zero-1})
we have $\mathbf{Ric}_{\tilde{g}_{\varepsilon}}-\tilde{\mathbf{Ric}}_{\varepsilon}\to0$
uniformly on compact sets and so 
\begin{equation}
\mathbf{Ric}_{\varepsilon}-\tilde{\mathbf{Ric}}_{\varepsilon}\to0\quad\text{uniformly on compact sets.}\label{eq:R-R*rho_eps goes to zero}
\end{equation}

Now we define $A_{\varepsilon}:=\mathbf{Ric}_{\varepsilon}-\left(n-1\right)\kappa g_{\varepsilon}$
and $A:=\mathbf{Ric}-\left(n-1\right)\kappa g$. Clearly
$A_{\varepsilon}-\tilde{A}_{\varepsilon}\to0$ uniformly on compact
sets and thus (for $\varepsilon$ small enough) 
\[
\left|A_{\varepsilon}(X,X)-\tilde{A}_{\varepsilon}(X,X)\right|\leq c\,\sup_{i,j\leq n}\sup_{p\in K}\left|A_{ij}(p)-\tilde{A}_{\varepsilon,ij}(p)\right|<\delta(n-1)
\]
for all $X\in TM|_{K}$ with $\left\Vert X\right\Vert _{h}\leq C$.
So if we can show that $\tilde{A}_{\varepsilon}(X,X)\geq0$ for all
$X\in TM|_{K}$ with $\left\Vert X\right\Vert _{h}\leq C$ and $g_{\varepsilon}(X,X)=-1$
the claim follows. 

As in Lem.\ \ref{lem:Riem approx ricci} it now suffices to show
this for every term of $\tilde{A}_{\varepsilon}$ of the form $\zeta_{\alpha}\psi_{\alpha}^{*}((\tilde{\chi}_{\alpha}A_{ij})*\rho_{\varepsilon})$.
Again we may assume $M=\mathbb{R}^{n}$ and $\tilde{A}_{\varepsilon}=A*\rho_{\varepsilon}$.
Now choose $\varepsilon_{0}$ such that $\left|g_{\varepsilon}(X,X)-g\left(X,X\right)\right|<\frac{1}{2}$
for all $X\in TM|_{K}$ with $\left\Vert X\right\Vert _{h}\leq C$
and all $\varepsilon<\varepsilon_{0}$. Since $g$ is uniformly continuous
on $K$ there exists some $\varepsilon_{0}>r>0$ such that for any
$p,x\in K$ with $\left\Vert x-p\right\Vert _{h}<r$ and any $X_{p}\in T_{p}M=\mathbb{R}^{n}$
with $\left\Vert X_{p}\right\Vert _{h}\leq C$ we have $\left|g(p)(X_{p},X_{p})-g(x)(X_{p},X_{p})\right|<\frac{1}{2}$.
This implies that for any $p\in K$ and $X_{p}\in\mathbb{R}^{n}$
with $\left\Vert X_{p}\right\Vert _{h}\leq C$ and $g_{\varepsilon}(p)(X_{p},X_{p})=-1$
the constant vector field $\tilde{X}:\, x\mapsto X_{p}$ is $g$ timelike
on on the open ball $B_{p}(r)$ and thus by our assumption $A(\tilde{X},\tilde{X})=\mathbf{Ric}(\tilde{X},\tilde{X})-\left(n-1\right)\kappa g(\tilde{X},\tilde{X})\geq0$
almost everywhere on $B_{p}(r)$. So for $\varepsilon<r$ we get
\[
\tilde{A}_{\varepsilon}(p)(X_{p},X_{p})=(A*\rho_{\varepsilon})(p)(X_{p},X_{p})=(A(\tilde{X},\tilde{X})*\rho_{\varepsilon})(p)\geq0,
\]
since $\rho_{\varepsilon}\geq0$ and $\mathrm{supp}\rho_{\varepsilon}\subset B_{0}(\varepsilon)$.\end{proof}
\begin{rem}
Note that the condition $\left\Vert X\right\Vert _{h}<C$ in the inequality
(\ref{eq:ricci for eps metric}) was not necessary in the Riemannian
case (see Lem.\ \ref{lem:Riem approx ricci}) since $g$ itself was
Riemannian, but is vital for Lorentzian metrics and without it, the
result is probably not true: For example, if $M=\mathbb{R}^{3}$ with
$g=\mathrm{diag}(-1,1,1)$ and $g_{\varepsilon}=\mathrm{diag}(-1-\varepsilon x^{2}y^{2}z^{2},1,1)$
then $g_{\varepsilon}\to g$ even in $\mathcal{C}^{\infty}$ and $\mathbf{Ric}_{g}(X,X)\geq0$,
but for $p=(1,1,1)\in M$, $N\in\mathbb{N}$ and any $\varepsilon>0$
there exists $X=X(N,\varepsilon)\in T_{p}M=\mathbb{R}^{3}$ such that
$g_{\varepsilon}(p)(X,X)=-1$ but still
\[
\mathbf{Ric}_{g_{\varepsilon}}(p)(X,X)<-N.
\]
However, these $X(N,\varepsilon)$ do not satisfy $\left\Vert X(N,\varepsilon)\right\Vert _{h}<C$
for any $C>0$ independent of $\varepsilon$ and $N$. A straightforward
calculation gives
\[
\mathbf{Ric}_{g_{\varepsilon}}(p)=\frac{1}{\left(1+\varepsilon\right)^{2}}\left(\begin{array}{ccc}
\left(1+\varepsilon\right)2\varepsilon & 0 & 0\\
0 & -\varepsilon & -\varepsilon\left(2+\varepsilon\right)\\
0 & -\varepsilon\left(2+\varepsilon\right) & -\varepsilon
\end{array}\right).
\]
Now let $X=(x,y,y)\in T_{p}M$ and demand $-1=g_{\varepsilon}(p)(X,X)=\left(-1-\varepsilon\right)x^{2}+2y^{2}$.
Then
\begin{align*}
\mathbf{Ric}_{g_{\varepsilon}}(p)(X,X) & =\frac{2\varepsilon}{\left(1+\varepsilon\right)^{2}}\left[1-(1+\varepsilon)y^{2}\right]
\end{align*}
which diverges to $-\infty$ as $y\to\infty$ for any fixed $\varepsilon$.\end{rem}
\begin{lem}
\label{lem: mean curvature}Let $g\in\mathcal{C}^{1,1}$ and assume
that the mean curvature of $\Sigma\subset M$ is bounded from above
by $\beta$. Then there exist approximations $g_{\varepsilon}$ such
that for any compact set $A\subset\Sigma$ and $\eta>0$ there exists
$\varepsilon_{0}$ such that $H_{\varepsilon}|_{A}<\beta+\eta$
for all $\varepsilon<\varepsilon_{0}$.\end{lem}
\begin{proof}
Since $H=\mathrm{tr}_{g|_{T\Sigma}}\, S_{\mathbf{n}}$ (see Rem.\
\ref{rem: conventions for shape op and mean curv}) and the Christoffel
symbols of $g_{\varepsilon}$ converge to those of $g$ uniformly
on compact sets it suffices to show that the the $g_{\varepsilon}$
unit normal vector field $\mathbf{n}_{\varepsilon}$ to $\Sigma$
converges to $\mathbf{n}$ in $\mathcal{C}^{1}$. Because $\Sigma$
is a smooth hypersurface it is locally given as the zero set of a
submersion $f:U\to\mathbb{R}^{n-1}$ and hence 
\begin{equation}
\mathbf{n}_{\varepsilon}=\frac{\mathrm{grad}_{\varepsilon}f}{\left|\mathrm{grad}_{\varepsilon}f\right|_{g_{\varepsilon}}}\to\frac{\mathrm{grad}f}{\left|\mathrm{grad}f\right|_{g}}=\mathbf{n}\label{eq:n als grad f}
\end{equation}
in $\mathcal{C}^{1}$, proving the claim.
\end{proof}
We need two further properties of this approximations.
\begin{prop}
\label{prop: f}Let $K\subset TM$ be compact and $T>0$ such that
all $g$-geodesics starting in $K$ exist for all $t\leq T$. Then
there exists $\varepsilon_{0}>0$ such that for all $\varepsilon_{0}\geq\varepsilon\geq0$
every $g_{\varepsilon}$-geodesic starting in $K$ exists until at
least time $T$ and the function
\begin{align*}
f:\left[0,\varepsilon_{0}\right]\times\left[0,T\right]\times K & \to TM\\
\left(\varepsilon,t,v\right) & \mapsto\dot{\gamma}_{v}^{\varepsilon}(t),
\end{align*}
where $\gamma_{v}^{\varepsilon}$ denotes the $g_{\varepsilon}$-geodesic
with $\dot{\gamma}_{v}^{\varepsilon}(0)=v$, is continuous. \end{prop}
\begin{proof}
This follows from a local argument using a standard result on the
comparison of solutions to ODE (\citep[10.5.6 and 10.5.6.1]{Dieudonne_vol1}):
Note that the $\Gamma_{g,ij}^{k}$ are locally Lipschitz continuous,
the $\Gamma_{\varepsilon,ij}^{k}(p)$ depend smoothly on $\varepsilon$
and $p$ for $\varepsilon>0$ and $\Gamma_{\varepsilon,ij}^{k}\to\Gamma_{g,ij}^{k}$
locally uniformly for $\varepsilon\to0$. Given any $v\in K$ we choose chart domains $U_i \subset TM$ ($i=0,\dots ,m$) covering $\dot{\gamma}_v([0,T])$ and times $t_i$ such that $\dot{\gamma}_v([t_i,t_{i+1}])\subset U_i$. Let $k^i$ be an upper bound for the Lipschitz constants of the derivatives of $g$ and $g_\varepsilon$ in $U_i$ and $\alpha^i$ be chosen such that $|\Gamma_{\varepsilon,lj}^{k}-\Gamma_{g,lj}^{k}|<\alpha^i$ on $U_i$. Then
by \citep[10.5.6 and 10.5.6.1]{Dieudonne_vol1} for any $\tilde v\in U_0\cap K$ with $\left\Vert v^0-\tilde v^0 \right\Vert _{e}<\mu^0$ one has that for $\mu^0$ and $\alpha^0$ sufficiently small $\dot{\gamma}^{\varepsilon,0}$ exists until at least $t_1$ and  $\left\Vert \dot{\gamma}^0_v(t)-\dot{\gamma}^{\varepsilon,0}_{\tilde v}(t)\right\Vert _{e} \leq \mu e^{tk}+\alpha \frac{e^{tk}-1}{k}$ for all $t\in[0,t_1]$. Continuing this in $U_1$ (with initial data $\dot{\gamma}_v(t_1)$ and $\dot{\gamma}^{\varepsilon}_{\tilde v}(t_1)$, which will be close if $v$ and $\tilde v$ were), $U_2$ and so forth gives the claim.
\end{proof}
\begin{defn}
[Unit normal bundle] We write $S^{+}N\Sigma$ (or sometimes also
$S_{0}^{+}N_{0}\Sigma$) for the (future) unit normal bundle to $\Sigma$,
i.e.\, 
\[
S^{+}N\Sigma:=\left\{ v\in TM|_{\Sigma}:\, v\, f.p.,\, g(v,w)=0\,\forall w\in T_{\pi(v)}\Sigma\:\text{and}\, g(v,v)=-1\right\} \subset TM|_{\Sigma}
\]
and analogously $S_{\varepsilon}^{+}N_{\varepsilon}\Sigma$ for the
(future) unit normal bundle to $\Sigma$ w.r.t.\ the metric $g_{\varepsilon}$.
For any $A\subset\Sigma$ we further define $S^{+}NA\equiv S_{0}^{+}N_{0}A:=\left\{ v\in S^{+}N\Sigma:\,\pi(v)\in A\right\} $
and analogously $S_{\varepsilon}^{+}N_{\varepsilon}A$.
\end{defn}
For compact $A\subset\Sigma$ each $S_{\varepsilon}^{+}N_{\varepsilon}A$
is compact for any $\varepsilon\geq0$ (since the respective future
pointing unit normal vector fields $\mathbf{n}_{\varepsilon}$ are
continuous and $S_{\varepsilon}^{+}N_{\varepsilon}A=\mathbf{n}_{\varepsilon}(A)$
by definition). The following lemma shows that this remains true for
their union over $0\leq\varepsilon\leq\varepsilon_{0}$.
\begin{lem}
\label{lem: starting vectors ompact}Let $A\subset\Sigma$ be compact.
Then for any neighborhood $U$ of $S^{+}NA$ in $TM|_{\Sigma}$ there
exists $\varepsilon_{0}(U,A)>0$ such that 
\[
\bigcup_{0\leq\varepsilon\leq\varepsilon_{0}}S_{\varepsilon}^{+}N_{\varepsilon}A\subset U\subset TM|_{\Sigma}
\]
and is compact.\end{lem}
\begin{proof}
By definition of the unit normal bundles $S_{\varepsilon}^{+}N_{\varepsilon}A$
we have $\bigcup_{0\leq\varepsilon\leq\varepsilon_{0}}S_{\varepsilon}^{+}N_{\varepsilon}A=n([0,\varepsilon_{0}],A)$
where $n:[0,1]\times\Sigma\to TM|_{\Sigma}$ is defined by $n(\varepsilon,p):=\mathbf{n}_{\varepsilon}(p)$,
so the assertion follows from continuity of $n$ (which in turn follows
directly from (\ref{eq:n als grad f})).  
\end{proof}

\subsection{\label{sub:The-cut-locus}The cut locus of $\Sigma$ has measure
zero}

As a further preparation we will now show that for an acausal, spacelike,
FCC hypersurface $\Sigma$ in a globally hyperbolic spacetime with
$\mathcal{C}^{1,1}$-metric the (future) cut locus $\mathrm{Cut}^{+}(\Sigma)\subset M$
has measure zero. This will be vital in the proof of Lem.\ \ref{lem:volume convergence}.
\begin{defn}
[Cut function]Let $\left(M,g\right)$ with $g\in\mathcal{C}^{1,1}$
be globally hyperbolic and $\Sigma\subset M$ be an acausal, spacelike,
FCC hypersurface. The function 
\begin{align*}
s_{\Sigma}^{+}:S^{+}N\Sigma & \to\bar{\mathbb{R}}\\
s_{\Sigma}^{+}(v) & :=\sup\left\{ t>0:\,\tau_{\Sigma}(\gamma_{v}(t))=L(\gamma_{v}|_{\left[0,t\right]})\right\} 
\end{align*}
is called the cut function. 
\end{defn}
We first show measurability of the cut function. 
\begin{lem}
[Measurability of the cut function]\label{lem: cut function measurable}The
cut function is measurable with respect to the completion $\mathcal{B}_{\mu_{g}}$
of the Borel-$\sigma$-algebra of $S^{+}N\Sigma$ w.r.t.\ the measure
$\mu_{g}$ induced by the metric $g$.\end{lem}
\begin{proof}
To begin with we rewrite the cut function in a form that makes it
possible to use Prop.\ \ref{prop:measurable sup}. Define the set-valued
map $F:S^{+}N\Sigma\to\mathcal{P}(\mathbb{R})$ by
\[
F(v):=\left\{ t\in\mathbb{R}:\,\left(v,t\right)\in\mathcal{D}\,\mathrm{and}\,\tau_{\Sigma}(\gamma_{v}(t))=L(\gamma_{v}|_{\left[0,t\right]})\right\} ,
\]
where $\mathcal{D}$ denotes the maximal domain of definition of the
flow of the (normal) exponential map. Note that $\mathcal{D}$ is
open. Then (using Prop.\ \ref{prop: time sep continuous})
\[
s_{\Sigma}^{+}(v)=\sup\left\{ t:\, t\in F(v)\right\} .
\]
Since $\mathbb{R}$ is a Suslin space (see Ex.\ \ref{ex: R suslin})
and $f=\mathrm{pr}_{\mathbb{R}}:S^{+}N\Sigma\times\mathbb{R}\to\mathbb{R}$
is continuous (so in particular measurable) it only remains to show
that
\[
\mathrm{graph}(F)=\left\{ \left(v,t\right)\in\mathcal{D}:\,\tau_{\Sigma}(\gamma_{v}(t))=L(\gamma_{v}|_{\left[0,t\right]})\right\} 
\]
is measurable. This in turn follows immediately if we can show that
both the map $\left(v,t\right)\mapsto\tau_{\Sigma}(\gamma_{v}(t))$
and $(v,t)\mapsto L(\gamma_{v}|_{\left[0,t\right]})$ are continuous
on $\mathcal{D}$. The first continuity follows from continuity of
$\tau_{\Sigma}$ on $M$ (see Lem.\ \ref{prop: time sep continuous})
and continuity of $\left(v,t\right)\mapsto\gamma_{v}(t)$ on $\mathcal{D}$
(by continuous dependence of ODE solutions on the initial data). For
the second one, note that $\left(v,t\right)\mapsto\int_{0}^{t}g(\gamma_{v}(\tau))(\dot{\gamma}_{v}(\tau),\dot{\gamma}_{v}(\tau))d\tau$
is continuous because the integrand is. \end{proof}
\begin{rem}
Note that $\mathcal{B}_{\mu_{g}}$ is actually \ independent of $g$:
If $\tilde{g}$ is any other semi-Riemannian metric on $M$
then $\mathcal{B}_{\mu_{g}}=\mathcal{B}_{\mu_{\tilde{g}}}$because
the Borel sets of measure zero are the same for $\mu_{g}$ and $\mu_{\tilde{g}}$
as locally any such measure is given by the Lebesque measure multiplied
by a positive function (cf.\ \citep[16.22.2]{vol3_Dieudonne}).

Also, for smooth metrics measurability is a direct consequence of
lower semi-continuity of the cut function, but the proof of lower
semi-continuity uses the characterization of the cut points as either
conjugate points or meeting points of two maximizing geodesics (see,
e.g., \citep[Prop.~9.7]{BEE96}), which one does not have in the $\mathcal{C}^{1,1}$
case and it is unclear whether lower semi-continuity even remains
true for $\mathcal{C}^{1,1}$-metrics.\end{rem}
\begin{defn}
[Cut locus]The tangential (future) cut locus is defined as
\[
\mathrm{Cut}_{T}^{+}(\Sigma):=\left\{ s_{\Sigma}^{+}(v)v:\, v\in S^{+}N\Sigma\:\:\mathrm{and}\,\left(v,s_{\Sigma}^{+}(v)\right)\in\mathcal{D}\right\} \subset N\Sigma.
\]
The (future) cut locus is defined as the image of the tangential
cut locus under the normal exponential map:
\[
\mathrm{Cut}^{+}(\Sigma):=\exp^{N}(\mathrm{Cut}_{T}^{+}(\Sigma)).
\]
\end{defn}
\begin{prop}
\label{prop:cut locus has measure zero}Let $\left(M,g\right)$ with
$g\in\mathcal{C}^{1,1}$ be globally hyperbolic and $\Sigma\subset M$
be an acausal, spacelike, FCC hypersurface. Then the future cut locus
$\mathrm{Cut}^{+}(\Sigma)\subset M$ has measure zero.\end{prop}
\begin{proof}
First note that $S^{+}N\Sigma\times(0,\infty)\cong N\Sigma\setminus\left\{ 0\right\} $
via $\left(v,t\right)\mapsto tv$. Using this identification we have
\[
\mathrm{Cut}_{T}^{+}(\Sigma)=\left\{ \left(v,s_{\Sigma}^{+}(v)\right):v\in S^{+}N\Sigma\:\:\mathrm{and}\,\left(v,s_{\Sigma}^{+}(v)\right)\in\mathcal{D}\right\} =\mathrm{graph}(s_{\Sigma}^{+})\cap\mathcal{D}.
\]
So from measurability of the cut function (Lem.\ \ref{lem: cut function measurable})
and Prop.\ \ref{prop: graph is measurable} and Prop.\ \ref{prop: graph has measure zero}
from the appendix, we obtain that the tangential cut locus $\mathrm{Cut}_{T}^{+}(\Sigma)\subset N\Sigma$
has measure zero.

Now the normal exponential map $\exp^{N}:N\Sigma\to M$ is a locally
Lipschitz continuous map from the $n-1+1=n$-dimensional manifold
$N\Sigma$ to the $n$-dimensional manifold $M$, hence its chart
representations (with relatively compact domains) can be extended
to Lipschitz continuous maps from $\mathbb{R}^{n}\to\mathbb{R}^{n}$.
Using a compact exhaustion $K_{n}$ of $N\Sigma$ and covering
each $K_{n}$ by finitely many charts (with relatively compact domains)
we see from Prop.\ \ref{prop: L stetig maps zero measure to zero} that $\exp^{N}(K_{n}\cap\mathrm{Cut}_{T}^{+}(\Sigma))$ has
measure zero. Thus $\mathrm{Cut}^{+}(\Sigma)=\bigcup_{n}\exp^{N}(K_{n}\cap\mathrm{Cut}_{T}^{+}(\Sigma))$
has measure zero.
\end{proof}

\subsection{\label{sub:The-comparison-manifolds}The comparison manifolds}

For any given $\kappa,\beta$ we use the comparison manifolds $M_{\kappa,\beta}$
defined in \citep[Sec.~4.2]{TG}. To make this work more self contained
we will briefly review their definition and properties.

These comparison manifolds were constructed to satisfy $CCC(\kappa,\beta)$
with equality in both the Ricci as well as the mean curvature estimates
and are given by certain warped products $M_{\kappa,\beta}=\left(a_{\kappa,\beta},b_{\kappa,\beta}\right)\times N_{\kappa,\beta}$
for $0\in(a_{\kappa,\beta},b_{\kappa,\beta})\subset\mathbb{R}$, where
$N_{\kappa,\beta}$ is the $(n-1)$-dimensional simply connected Riemannian manifold
with constant sectional curvature of $0$, $1$, or $-1$, depending
on $\kappa$ and $\beta$ (so either $\mathbb{R}^{n-1}$, the unit
sphere $S^{n-1}$, or hyperbolic space $H^{n-1}$), with metric
\[
g_{\kappa,\beta}=-dt^{2}+f_{\kappa,\beta}(t)^{2}h,
\]
where $h$ denotes the standard Riemannian metric on $N_{\kappa,\beta}$
and $f_{\kappa,\beta}:\left(a_{\kappa,\beta},b_{\kappa,\beta}\right)\to\mathbb{R}$
is a positive smooth function. 

The warping functions for each pair $\kappa,\beta$ are summarized
in the table below, which is based on \citep[Table 1]{TG}, but we
use that the mean curvature $H_{0}$ of the hypersurface $\Sigma_{\kappa,\beta}:=\{0\}\times N_{\kappa,\beta}\subset M_{\kappa,\beta}$
is equal to $\beta$ to express their constant $b$ in terms of $\beta$
and we also include the respective constants $b_{\kappa,\beta}$ that
specify the upper bound of the interval where $f_{\kappa,\beta}^{2}>0$.
Also note that our base manifold is assumed to be $n$-dimensional
(whereas it is $(n+1)$-dimensional in \citep{TG}) and for notational
simplicity some of the $f_{\kappa,\beta}$ listed in Table \ref{tab:Warping-functions-for}
are strictly negative instead of positive. (Of course one may replace
them with $-f_{\kappa,\beta}$ to obtain positive warping functions.)
\begin{table}
\begin{centering}
\begin{tabular}{cccccc}
$\kappa$ & $\beta$ & $N_{\kappa,\beta}$ & $b$ & $f_{\kappa,\beta}(t)$ & $b_{\kappa,\beta}$\tabularnewline
\midrule
\midrule 
$\kappa<0$ & $\frac{\left|\beta\right|}{(n-1)\sqrt{\left|\kappa\right|}}<1$ & $S^{n-1}$ & $\tanh^{-1}(\frac{\beta}{(n-1)\sqrt{\left|\kappa\right|}})$ & $\frac{1}{\sqrt{\left|\kappa\right|}}\cosh(\sqrt{\left|\kappa\right|}t+b)$ & $\infty$\tabularnewline
\midrule
 & $\frac{\left|\beta\right|}{(n-1)\sqrt{\left|\kappa\right|}}=1$ & $\mathbb{R}^{n-1}$ & $0$ & $\exp(\mathrm{sgn}(\beta)\sqrt{\left|\kappa\right|}t)$ & $\infty$\tabularnewline
\midrule
 & $\frac{\beta}{(n-1)\sqrt{\left|\kappa\right|}}>1$ & $H^{n-1}$ & $\coth^{-1}(\frac{\beta}{(n-1)\sqrt{\left|\kappa\right|}})$ & $\frac{1}{\sqrt{\left|\kappa\right|}}\sinh(\sqrt{\left|\kappa\right|}t+b)$ & $\infty$\tabularnewline
\midrule
 & $\frac{\beta}{(n-1)\sqrt{\left|\kappa\right|}}<-1$ & $H^{n-1}$ & $\coth^{-1}(\frac{\beta}{(n-1)\sqrt{\left|\kappa\right|}})$ & $\frac{1}{\sqrt{\left|\kappa\right|}}\sinh(\sqrt{\left|\kappa\right|}t+b)$ & $-\frac{b}{\sqrt{\left|\kappa\right|}}$\tabularnewline
\midrule
\midrule 
$\kappa=0$ & $\beta=0$ & $\mathbb{R}^{n-1}$ & $0$ & $1$ & $\infty$\tabularnewline
\midrule
 & $\beta>0$ & $H^{n-1}$ & $\frac{n-1}{\beta}$ & $t+b$ & $\infty$\tabularnewline
\midrule
 & $\beta<0$ & $H^{n-1}$ & $\frac{n-1}{\beta}$ & $t+b$ & $-\frac{n-1}{\beta}$\tabularnewline
\midrule
\midrule 
$\kappa>0$ & $\beta\in\mathbb{R}\setminus\{0\}$ & $H^{n-1}$ & $\cot^{-1}(\frac{\beta}{(n-1)\sqrt{\kappa}})$ & $\frac{1}{\sqrt{\kappa}}\sin(\sqrt{\kappa}t+b)$ & $\frac{-b+\frac{\pi}{2}(1+\mathrm{sgn}(\beta))}{\sqrt{\kappa}}$\tabularnewline
\midrule
 & $\beta=0$ & $H^{n-1}$ & $\frac{\pi}{2}$ & $\frac{1}{\sqrt{\kappa}}\cos(\sqrt{\kappa}t)$ & $\frac{\pi}{2\sqrt{\kappa}}$\tabularnewline
\bottomrule
\end{tabular}
\par\end{centering}

\medskip{}

\caption{\label{tab:Warping-functions-for}Warping functions for different
values of $\kappa,\beta$.}
\end{table}

Since $\lim_{t\nearrow b_{\kappa,\beta}}f_{\kappa,\beta}(t)=0$ if
$b_{\kappa,\beta}<\infty$ we may define continuous functions $\tilde{f}_{\kappa,\beta}:[0,\infty)\to\mathbb{R}$
by extending $f_{\kappa,\beta}$ by zero if necessary, so 
\begin{equation}
\tilde{f}_{\kappa,\beta}(t):=\begin{cases}
f_{\kappa,\beta}(t) & t<b_{\kappa,\beta}\\
0 & t\geq b_{\kappa,\beta}.
\end{cases}\label{eq:def of f tilde}
\end{equation}
We now investigate some of the circumstances under which convergence
of $\kappa\nearrow\kappa_{0}$ and $\beta\searrow\beta_{0}$ (since
the approximating metrics we will use satisfy Lem.\ \ref{lem: ricci}
and Lem.\ \ref{lem: mean curvature}) implies pointwise convergence
of the corresponding $\tilde{f}_{\kappa,\beta}$ or at least of the
functions $t\mapsto\frac{\tilde{f}_{\kappa,\beta}(t)}{f_{\kappa,\beta}(0)}$.
Clearly the map $\left(\kappa,\beta\right)\mapsto\tilde{f}_{\kappa,\beta}$
is continuous (w.r.t.\ pointwise convergence) on $\left(\mathbb{R}_{>0}\times\mathbb{R}\setminus\{0\}\right)\cup\{(\kappa,\beta)\in\mathbb{R}^{2}:\,\kappa<0\,\mathrm{and}\,\left|\beta\right|\neq(n-1)\sqrt{|\kappa|}\}$.
For the remaining points, simple calculations show the following:
\begin{lem}
\label{lem:convergence of f tilde}If either
\begin{enumerate}
\item $\kappa_{0}=0$, $\beta_{0}\neq0$ and $\kappa\nearrow0$ and $\beta\to\beta_{0}$,
\item $\kappa_{0}=0$, $\beta_{0}=0$ and $\kappa\nearrow0$ and $\beta:=(n-1)\sqrt{|\kappa|}\searrow0$,
or
\item $\kappa_{0}>0$, $\beta_{0}=0$ and $\kappa\to\kappa_{0}$ and $\beta\searrow0$,
\item $\kappa_{0}<0,$ $\beta_{0}=(n-1)\sqrt{|\kappa_{0}|}$ and $\kappa\nearrow\kappa_{0}$
and $\beta:=(n-1)\sqrt{|\kappa|}\searrow\beta_{0}$
\end{enumerate}
then $\tilde{f}_{\kappa,\beta}\to\tilde{f}_{\kappa_{0},\beta_{0}}$
pointwise. And if
\[
\kappa_{0}<0,\,\beta_{0}=-(n-1)\sqrt{|\kappa_{0}|}\:\mathrm{and}\:\kappa\nearrow\kappa_{0}\:\mathrm{and}\:\beta\searrow\beta_{0}
\]
then $\tilde{f}_{\kappa,\beta}\not\to\tilde{f}_{\kappa_{0},\beta_{0}}$
but still $\frac{\tilde{f}_{\kappa,\beta}(t)}{f_{\kappa,\beta}(0)}\to\frac{\tilde{f}_{\kappa_{0},\beta_{0}}(t)}{f_{\kappa_{0},\beta_{0}}(0)}$
for all $t\geq0$. 

Furthermore, for any $\kappa\leq0$ and $\beta\in\mathbb{R}$ one
has $\left|\frac{\tilde{f}_{\kappa,\beta}(t)}{f_{\kappa,\beta}(0)}\right|\leq\max\left\{ \left|\frac{\tilde{f}_{\kappa,\beta}(T)}{f_{\kappa,\beta}(0)}\right|,1\right\} $
for all $t\leq T$ because $\left|\tilde{f}_{\kappa,\beta}\right|$
is monotone or convex and for $\kappa>0$ and $\beta\in\mathbb{R}$ we have
$\left|\tilde{f}_{\kappa,\beta}(t)\right|\leq \frac{1}{\sqrt{\kappa}}$ for all $t\in\mathbb{R}$.
\end{lem}
This convergence result can be used to show convergence of areas and
volumes of future spheres and balls in $M_{\kappa,\beta}$ above a
subset of $\Sigma_{\kappa,\beta}=\{0\}\times N_{\kappa,\beta}\subset M_{\kappa,\beta}$
(which is an acausal, spacelike, FCC hypersurface in $M_{\kappa,\beta}$).

\begin{defn}
[Future spheres and balls] For any $t>0$ and $A\subset\Sigma$ we
define the spheres $S_{A}^{+}(t)$ and balls $B_{A}^{+}(t)$ of time
$t$ above $A$ by
\begin{align*}
S_{A}^{+}(t): & =\{p\in I^{+}(\Sigma):\,\exists q\in A\,\mathrm{with}\, d(q,p)=\tau_{\Sigma}(p)=t\}\:\mathrm{and}\\
B_{A}^{+}(t): & =\bigcup_{s\in\left(0,t\right)}S_{A}^{+}(s)
\end{align*}

\end{defn}
Using these definitions and Lem.\ \ref{lem:convergence of f tilde}
we show:
\begin{cor}
For any $\left(\kappa,\beta\right)\in\mathbb{R}^{2}$ there exist
sequences $\delta_{n}>0$ and $\eta_{n}>0$ converging to zero such
that 
\begin{equation}
\frac{\mathrm{area}_{\kappa-\delta_{n},\beta+\eta_{n}}S_{B_{n}}^{+}(t)}{\mathrm{area}_{\kappa-\delta_{n},\beta+\eta_{n}}B_{n}}\to\frac{\mathrm{area}_{\kappa,\beta}S_{B}^{+}(t)}{\mathrm{area}_{\kappa,\beta}B}\label{eq:area convergence in comparison space}
\end{equation}
and
\begin{equation}
\frac{\mathrm{vol}_{\kappa-\delta_{n},\beta+\eta_{n}}B_{B_{n}}^{+}(t)}{\mathrm{area}_{\kappa-\delta_{n},\beta+\eta_{n}}B_{n}}\to\frac{\mathrm{vol}_{\kappa,\beta}B_{B}^{+}(t)}{\mathrm{area}_{\kappa,\beta}B}\label{eq:volume convergence in comparison space}
\end{equation}
for all $t>0$ and any measurable $B_{n}\subset\Sigma_{\kappa-\delta_{n},\beta+\eta_{n}}$
and measurable $B\subset\Sigma_{\kappa,\beta}$.\end{cor}
\begin{proof}
The first statement is an immediate consequence of the previous Lemma
and $\frac{\mathrm{area}_{\kappa,\beta}S_{B}^{+}(t)}{\mathrm{area}_{\kappa,\beta}B}=(\frac{\tilde{f}_{\kappa,\beta}(t)}{f_{\kappa,\beta}(0)})^{n-1}$
for any $\kappa,\beta$ (see \citep[eq.~(15)]{TG} and note that $S_{B}^{+}(t)=\emptyset$
for $t\geq b_{\kappa,\beta}$). For (\ref{eq:volume convergence in comparison space})
note that and $\mathrm{vol}_{\kappa,\beta}B_{B}^{+}(t)=\int_{0}^{t}\mathrm{area}_{\kappa,\beta}S_{B}^{+}(\tau)d\tau$
and that we may apply dominated convergence since for $\kappa\leq0$
and $\beta\in\mathbb{R}$ one has 
\[
\left|\frac{\tilde{f}_{\kappa-\delta_{n},\beta+\eta_{n}}(\tau)}{f_{\kappa-\delta_{n},\beta+\eta_{n}}(0)}\right|\leq\max\left\{ \left|\frac{\tilde{f}_{\kappa-\delta_{n},\beta+\eta_{n}}(T)}{f_{\kappa-\delta_{n},\beta+\eta_{n}}(0)}\right|,1\right\} \to\max\left\{ \left|\frac{\tilde{f}_{\kappa,\beta}(T)}{f_{\kappa,\beta}(0)}\right|,1\right\} 
\]
for all $\tau\leq t$ and for $\kappa>0,\beta\in\mathbb{R}$ one has
$\left|\frac{\tilde{f}_{\kappa-\delta_{n},\beta+\eta_{n}}(\tau)}{f_{\kappa-\delta_{n},\beta+\eta_{n}}(0)}\right|\leq\frac{(\kappa-\delta_n)^{-\frac{1}{2}}}{\left|f_{\kappa-\delta_{n},\beta+\eta_{n}}(0)\right|}\to\frac{1}{\sqrt{\kappa}\left|f_{\kappa,\beta}(0)\right|}$.\end{proof}
\begin{rem}
The reason we only show this for specific sequences $\delta_{n}$
and $\eta_{n}$ lies in our somewhat incomplete treatment of the dependence
of $\nicefrac{\tilde{f}_{\kappa,\beta}}{\tilde{f}_{\kappa,\beta}(0)}$
on $\kappa,\beta$ in Lem.\ \ref{lem:convergence of f tilde}: While
it does seem reasonable that the result remains true for all such
sequences, that would require many additional cases of possible convergence
to be checked in Lem.\ \ref{lem:convergence of f tilde}, which is
rather tedious and completely unnecessary for the rest of this work.

\end{rem}

\subsection{\label{sub:Volume-Comparison}Volume Comparison}

We first need to show area and volume comparison statements for the
approximating metrics and to do so we need to define future spheres
that avoid the cut locus.
\begin{defn}
For $t>0$ let
\[
\EuScript S_{A}^{+}(t):=S_{A}^{+}(t)\setminus\mathrm{Cut}^{+}(\Sigma).
\]
Similarly, but using the approximations $g_{\varepsilon}$ (from section
\ref{sub:Desired-properties-of}), the $g_{\varepsilon}$-time separation
$\tau_{\varepsilon,\Sigma}$ and the $\varepsilon$-cut locus, we
define $\EuScript S_{\varepsilon,A}^{+}(t)$. 
\end{defn}
Using results from \citep{TG} we are now able to prove area and volume comparison
statements for the approximating metrics.
\begin{prop}
[Area comparison for approximations]\label{prop: area for approx}Let
$\kappa,\beta\in\mathbb{R}$, $g\in\mathcal{C}^{1,1}$ and assume
$\left(M,g,\Sigma\right)$ is globally hyperbolic and satisfies $CCC(\kappa,\beta)$.
Let $A\subset\Sigma$ be compact, $\eta,\delta>0$, $B\subset\Sigma_{\kappa-\delta,\beta+\eta}$
(with finite, non-zero area) and $T>0$ such that all timelike, f.d.\
unit speed $g$-geodesics starting in $A$ orthogonally to $\Sigma$
exist until at least $T$. Then there exists $\varepsilon_{0}>0$
(depending on $\eta,\delta,A,T$) such that for all $\varepsilon<\varepsilon_{0}$
the function 
\[
t\mapsto\frac{\mathrm{area}_{\varepsilon}\,\EuScript S_{\varepsilon,A}^{+}(t)}{\mathrm{area}_{\kappa-\delta,\beta+\eta}S_{B}^{+}(t)},
\]
is nonincreasing on $(0,T]$ if $T<b_{\kappa-\delta,\beta+\eta}$
or on $(0,b_{\kappa-\delta,\beta+\eta})$ if $T\geq b_{\kappa-\delta,\beta+\eta}$.\end{prop}
\begin{proof}
We would like to use \citep[Thm.~8]{TG}, however we have to argue
that the bounds on Ricci and mean curvature from Lemma \ref{lem: ricci}
and \ref{lem: mean curvature} are sufficient to show this for smooth
metrics (\citep{TG} requires global bounds while we only have them
on compact subsets of $TM$ respectively $\Sigma$). 

First we note that by compactness of $S^{+}NA$ there exists a neighborhood
$U$ of $S^{+}NA$ in $TM$ such that all $g$-geodesics starting
in $U$ exist until at least $T$. Then by Lem.\ \ref{lem: starting vectors ompact},
for $\varepsilon_{0}$ small the set $K_{\varepsilon_{0}}:=\bigcup_{0\leq\varepsilon\leq\varepsilon_{0}}S_{\varepsilon}^{+}N_{\varepsilon}A\subset TM$
is compact and contained in $U$, hence any $g$-geodesic starting
in it exists until $T$. So by Prop.\ \ref{prop: f} there exists
$\varepsilon_{0}>\tilde{\varepsilon}_{0}>0$ such that $f\left(\left[0,\tilde{\varepsilon}_{0}\right]\times\left[0,T\right]\times K_{\varepsilon_{0}}\right)=:\tilde{K}\subset TM$
is compact, in particular $h$-norm bounded by a constant $C$. 

From here the proof proceeds analogously to \citep[Thm.~8]{TG}. Let
$0<t_{1}<t_{2}<\min(T,b_{\kappa-\delta,\beta+\eta})$ and choose a
sequence of compact sets $K_{\varepsilon,j}\subset\EuScript S_{\varepsilon,A}^{+}(t_{2})$
with $\mathrm{area}\, K_{\varepsilon,j}\nearrow\mathrm{area}\,\EuScript S_{\varepsilon,A}^{+}(t_{2})$.
Now, as in \citep{TG}, we get sets
\[
K_{\varepsilon,j}(t):=\Phi_{\varepsilon,t-t_{2}}(K_{\varepsilon,j})\subset\EuScript S_{\varepsilon,K_{i}}^{+}(t),
\]
where $\Phi$ is the flow of $-\mathrm{grad}(\tau_{\varepsilon,\Sigma})$,
and
\[
\frac{d}{dt}\log(\mathrm{area}\, K_{\varepsilon,j}(t))=\frac{1}{\mathrm{area}\, K_{\varepsilon,j}(t)}\int_{K_{\varepsilon,j}(t)}H_{\varepsilon,t}(q)d\mu_{\varepsilon,t}(q).
\]
Next we show that for $\varepsilon$ small enough (depending on $\eta$,
$\delta$, $A$ and $\tilde{K}$) 
\begin{equation}
H_{\varepsilon,t}(q)\leq H_{\kappa-\delta,\beta+\eta}(t)\label{eq:mean curvature comparison}
\end{equation}
for all $q\in K_{\varepsilon,j}(t)$. This proceeds similarly to \citep[Thm.~7]{TG}:
For any $q\in\EuScript S_{\varepsilon,K_{j}}^{+}(t)$ the unique,
maximizing, unit speed $g_{\varepsilon}$-geodesic $\gamma_{\varepsilon}$
connecting $q$ to $\Sigma$ satisfies $\dot{\gamma}_{\varepsilon}(0)\in S_{\varepsilon}^{+}N_{\varepsilon}A$
and we have $H(\gamma(0))\leq\beta+\eta$ (by Lem.\ \ref{lem: mean curvature})
and $\dot{\gamma}_{\varepsilon}\subset\tilde{K}$ and hence $\mathbf{Ric}_{\varepsilon}(\dot{\gamma}_{\varepsilon}(s),\dot{\gamma}_{\varepsilon}(s))\geq n\,(\kappa-\delta)$
for all $s\in[0,t]$ (by Lem.\ \ref{lem: ricci}).  Note that this
is all that is needed to apply the Riccati comparison argument used
in the proof of Thm.\ 7 and it is the only place where the curvature
estimates enter the proof. So we get (\ref{eq:mean curvature comparison}).
The remainder of the proof is completely analogous to \citep[Thm.~8]{TG}.\end{proof}
\begin{prop}
[Volume comparison for approximations] \label{prop: volume for approx}Let
$\kappa,\beta\in\mathbb{R}$, $g\in\mathcal{C}^{1,1}$ and assume
$\left(M,g,\Sigma\right)$ is globally hyperbolic and satisfies $CCC(\kappa,\beta)$.
Let $A\subset\Sigma$ be compact, $\eta,\delta>0$, $B\subset\Sigma_{\kappa-\delta,\beta+\eta}$
(with finite, non-zero area) and $T>0$ such that all timelike, f.d.\
unit speed geodesics starting orthogonally to $A$ exist until at
least $T$. Then there exists $\varepsilon_{0}>0$ (depending on $\eta,\delta,A,T$)
such that for all $\varepsilon<\varepsilon_{0}$ the function 
\[
t\mapsto\frac{\mathrm{vol}_{\varepsilon}\, B_{\varepsilon,A}^{+}(t)}{\mathrm{vol}_{\kappa-\delta,\beta+\eta}B_{B}^{+}(t)}
\]
is nonincreasing on $(0,T]$ if $T<b_{\kappa-\delta,\beta+\eta}$
and on $(0,\infty)$ if $T\geq b_{\kappa-\delta,\beta+\eta}$.\end{prop}
\begin{proof}
For $T<b_{\kappa-\delta,\beta+\eta}$ this follows from the area comparison
result by using the coarea formula (see \citep[Thm.~9]{TG} for details).
Since $b_{\kappa-\delta,\beta+\eta}$ is defined by being the upper
bound of the maximal interval of positivity of the warping function
$f_{\kappa-\delta,\beta+\eta}$, we either have $b_{\kappa-\delta,\beta+\eta}=\infty$
or $b_{\kappa-\delta,\beta+\eta}<\infty$ and $\lim_{t\nearrow b_{\kappa-\delta,\beta+\eta}}f_{\kappa-\delta,\beta+\eta}(t)=0$.
In the second case an argument completely analogous to the area comparison
proof of \citep[Thm.~10]{TG} shows that $S_{\varepsilon,\Sigma}^{+}(t)=\emptyset$
for $t>b_{\kappa-\delta,\beta+\eta}$. Hence by the coarea formula
(see \citep[Prop.~3]{TG}) $t\mapsto\text{vol}_\varepsilon(B_{\varepsilon,A}^{+}(t))$
remains constant for $t>b_{\kappa-\delta,\beta+\eta}$ and thus $\frac{\mathrm{vol}_{\varepsilon}\, B_{\varepsilon,A}^{+}(t)}{\mathrm{vol}_{\kappa-\delta,\beta+\eta}B_{B}^{+}(t)}$
remains nonincreasing for $t>b_{\kappa-\delta,\beta+\eta}$.
\end{proof}
The plan is now to use Prop.\ \ref{prop: volume for approx} and
first let $\varepsilon\to0$ and then $\delta,\eta\to0$. To make
the proof more readable, we first show that $\mathrm{vol}_{\varepsilon}\, B_{\varepsilon,A}^{+}(t)\to\mathrm{vol}\, B_{A}^{+}(t)$
in a separate Lemma.
\begin{lem}
[Volume convergence]\label{lem:volume convergence} Let $g\in\mathcal{C}^{1,1}$,
assume $(M,g)$ is globally hyperbolic and let $\Sigma\subset M$
be an acausal, spacelike, FCC hypersurface. Let $A\subset\Sigma$
be compact with $\mu_{\Sigma}(\partial A)=0$ (where $\partial A$
is the boundary of $A$ in $\Sigma$) and $T>0$ such that all timelike,
f.d., unit speed geodesics starting orthogonally to $A$ exist until
at least $T$. Then for any $0<t\leq T$ we have
\[
\mathrm{vol}_{\varepsilon}\, B_{\varepsilon,A}^{+}(t)\to\mathrm{vol}\, B_{A}^{+}(t)
\]
for $\varepsilon\to0$. \end{lem}
\begin{proof}
From Prop.\ \ref{prop: f} and Lem.\ \ref{lem: starting vectors ompact}
it follows in a similar way as in the beginning of the proof of Prop.\
\ref{prop: area for approx} that $\bigcup_{0\leq\varepsilon\leq\varepsilon_{0}}\bar{B}_{\varepsilon,A}^{+}(T)$
(where $\bar{B}_{A}^{+}(t):=\{p\in I^{+}(\Sigma):\exists q\in A\, s.t.\,\tau_{\Sigma}(p)=d(p,q)\leq t\}$
for $t>0$) is contained in the compact set $K=f([0,\varepsilon_{0}]\times[0,T]\times K_{\varepsilon_{0}})$
(with $f$ as in Prop.\ \ref{prop: f}). Now fix $0<t\leq T$ then
$B_{\varepsilon,A}^{+}(t)\subset\bar{B}_{\varepsilon,A}^{+}(T)\subset K$
for all $0\leq\varepsilon\leq\varepsilon_{0}$ and $B_{\varepsilon,A}^{+}(t)\subset I_{\varepsilon}^{+}(\Sigma)\subset I^{+}(\Sigma)$
(note that we chose $g_{\varepsilon}$ such that $g_{\varepsilon}\prec g$).
So it only remains to show that the functions $\chi_{B_{\varepsilon,A}^{+}(t)}\,\sqrt{\left|\det g_{\varepsilon,ij}\right|}\to\chi_{B_{A}^{+}(t)}\,\sqrt{\left|\det g_{ij}\right|}$
pointwise almost everywhere on $K\cap I^{+}(\Sigma)$ and then apply
dominated convergence. This is clear for $\sqrt{\left|\det g_{\varepsilon,ij}\right|}$,
so we only have to look at the characteristic functions. 

First note that $\mu(\mathrm{Cut}^{+}(\Sigma))=0$ by Prop.\ \ref{prop:cut locus has measure zero},
so it suffices to show convergence a.e.\ on $(K\cap I^{+}(\Sigma))\setminus\mathrm{Cut}^{+}(\Sigma)$.
For any $p\in(K\cap I^{+}(\Sigma))\setminus\mathrm{Cut}^{+}(\Sigma)$
there exists a \emph{unique} (up to reparametrization) causal curve
$\gamma^{p}$ maximizing the distance from $p$ to $\Sigma$ (and
this curve is a geodesic starting orthogonally to $\Sigma$): Existence
follows from Lem.\ \ref{prop: time sep continuous} and if there
were two different maximizing geodesics none of them could be maximizing past $p$ (since locally
any maximizing timelike curve has to be an unbroken geodesic, see
\citep[Thm.~6]{Minguzzi_convexnbhdsLipConSprays}) and hence $p\in\mathrm{Cut}^{+}(\Sigma)$
by the definition of the cut locus. This allows us to split $(K\cap I^{+}(\Sigma))\setminus\mathrm{Cut}^{+}(\Sigma)$
into five (not necessarily mutually distinct) subsets:
\begin{enumerate}
\item we have $L\left(\gamma^{p}\right)<t$ and $\gamma^{p}(0)\in A^{\circ}$
(where $A^{\circ}$ is the interior of $A$ as a subset of $\Sigma$),
i.e.\ $p\in B_{A^{\circ}}^{+}(t)$, or
\item $L\left(\gamma^{p}\right)$ arbitrary and $\gamma^{p}(0)\notin A$,
in particular $p\notin\bar{B}_{A}^{+}(t)$, or
\item $L\left(\gamma^{p}\right)>t$ and $\gamma^{p}(0)$ arbitrary, so again
$p\notin\bar{B}_{A}^{+}(t)$, or
\item $L\left(\gamma^{p}\right)=t$ and $\gamma^{p}(0)\in A$, i.e.\ $p\in S_{A}^{+}(t)$,
or
\item $L\left(\gamma^{p}\right)\leq t$ and $\gamma^{p}(0)\in\partial A=A\setminus A^{\circ}$,
i.e.\ $p\in\bar{B}_{\partial A}^{+}(t)$
\end{enumerate}
We now show that in cases $\left(1\right)-\left(3\right)$ the characteristic
functions converge in $p$:

In case $\left(3\right)$ we have $L_{g}(\gamma^{p})>t$. But then
for $\varepsilon$ small this $\gamma^{p}$ is also $g_{\varepsilon}$
timelike and Lem.\ 4.2 from \citep{hawkingc11} gives that for any
small $\delta>0$ there exists $\varepsilon_0$ such that for
all $\varepsilon\leq\varepsilon_0$
\begin{equation}
\tau_{\varepsilon,\Sigma}(p)\geq L_{\varepsilon}(\gamma^{p})>L_{g}(\gamma^{p})-\delta>t.\label{eq:length estimate 1}
\end{equation}
 Thus $p\notin\bar{B}_{\varepsilon,A}^{+}(t)$ for $\varepsilon$
small.

Now for case $\left(1\right)$, let $p\in B_{A^{\circ}}^{+}(t)\subset B_{A}^{+}(t)$.
Let $\gamma_{\varepsilon}$ be a $g_{\varepsilon}$-geodesic between
$q_{\varepsilon}\in\Sigma$ and $p$ with $L_{\varepsilon}(\gamma_{\varepsilon})=\tau_{\varepsilon,\Sigma}(p)$.
From $q_{\varepsilon}\in J^{-}(p)\cap\Sigma$, it follows that $\gamma_{\varepsilon}\subset J^{-}(p)\cap J^{+}(J^{-}(p)\cap\Sigma)$
for all $\varepsilon$, which is compact by Rem.\ \ref{rem: J-capSig compact}.
This allows us to use Lem.\ 4.2 from \citep{hawkingc11} to obtain
that for any small $\delta>0$ there exists $\varepsilon_{0}$ such
that 
\begin{equation}
\tau_{\varepsilon,\Sigma}(p)=L_{\varepsilon}(\gamma_{\varepsilon})<L_{g}(\gamma_{\varepsilon})+\delta\leq\tau_{\Sigma}(p)+\delta.\label{eq:length estimate}
\end{equation}
 for all $\varepsilon\leq\varepsilon_{0}$. This shows that if
$\tau_{\Sigma}(p)<t$, then $\tau_{\varepsilon,\Sigma}(p)<t$ for
small $\varepsilon$. Now let $U\subset A^{\circ}$ be a neighborhood
of $\gamma^{p}(0)$ in $\Sigma$. It remains to show that $q_{\varepsilon}\in U\subset A^{\circ}$
for small $\varepsilon$. Assume the contrary and let $\gamma_{\varepsilon_{j}}$
be a subsequence with $q_{\varepsilon_{j}}\notin U$. By our limit
curve Lemma \ref{lem:Limit curve}, we may assume (after reparametrizing
and passing to a further subsequence) that $\gamma_{\varepsilon_{j}}$
converges to a causal curve $\tilde{\gamma}$ from $q:=\tilde{\gamma}(0)=\lim q_{\varepsilon_{j}}\notin U$
to $p$ with $L_{g}(\tilde{\gamma})\geq\limsup_{j\to\infty}L_{g}(\gamma_{\varepsilon_{j}})$.
Using (\ref{eq:length estimate}) and (\ref{eq:length estimate 1})
gives
\[
L_{g}(\tilde{\gamma})\geq\limsup_{j\to\infty}L_{g}(\gamma_{\varepsilon_{j}})\geq\limsup_{j\to\infty}\tau_{\varepsilon_{j},\Sigma}(p)-\delta\geq L_{g}(\gamma^{p})-2\delta=\tau_{\Sigma}(p)-2\delta
\]
for any $\delta>0$ and letting $\delta\to0$ shows that $\tilde{\gamma}$
is also maximizing the distance between $p$ and $\Sigma$, giving
a contradiction, since $\tilde{\gamma}(0)\neq\gamma^{p}(0)$ but $\gamma^{p}$
is the unique causal curve realizing the distance from $\Sigma$ to
$p$ by definition. Altogether, $p\in B_{\varepsilon,A^{\circ}}^{+}(t)\subset B_{\varepsilon,A}^{+}(t)$
for small enough $\varepsilon$.

Next we look at case $\left(2\right)$, i.e., $\gamma^{p}(0)\notin A$
(and thus $p\notin B_{A}^{+}(t)$). Let $U$ be a neighborhood of
$\gamma^{p}(0)$ in $\Sigma$ with $U\cap A=\emptyset$ (this exists
since $A$ is closed). By the argument presented when dealing with
case $\left(1\right)$, we have that for $\varepsilon$ small enough
$\gamma_{\varepsilon}(0)\in U$, hence not in $A$ and so $p\notin B_{\varepsilon,A}^{+}(t)$.

It remains deal with cases $\left(4\right)$ and $\left(5\right)$.
Here we show that both $S_{A}^{+}(t)$ and $\bar{B}_{\partial A}^{+}(t)$
are contained in sets of measure zero.

Regarding $S_{A}^{+}(t)$, let $\tilde{\mathbf{n}}$ be a $\mathcal{C}^{1,1}$-extension
of $\mathbf{n}$ to some small neighborhood $U$ of $A$ (in $M$)
and consider the map $h:\, p\mapsto\exp_{p}(t\tilde{\mathbf{n}}(p))$.
For $U$ small enough this is well defined on $U$ (by a standard
ODE argument) and because the exponential map is locally Lipschitz
continuous, this map is as well. Now since $S_{A}^{+}(t)\subset h(A)$,
$\mu(A)=0$ (because $A\subset\Sigma$), $A$ is compact and any Lipschitz
map from $\mathbb{R}^{n}\to\mathbb{R}^{n}$ maps sets of (Lebesgue-)measure
zero to sets of measure zero (see Prop.\ \ref{prop: L stetig maps zero measure to zero}
in the appendix), we have that $h(A)$ has measure zero.

Finally, for $\bar{B}_{\partial A}^{+}(t)$, note that $\partial A\subset A$
and hence all f.d., unit-speed, normal geodesics starting in $\partial A$
exist until at least $T\geq t$, so $\partial A\times[0,t]\subset\mathcal{D}$
and $\bar{B}_{\partial A}^{+}(t)\subset\exp^{N}\left([0,t]\cdot\partial A\right)$.
Now since $[0,t]\cdot\partial A\subset N\Sigma$ has measure zero
(because by assumption $\mu_{\Sigma}(\partial A)=0$) and is compact
(by compactness of $A$) and $\exp^{N}$ is locally Lipschitz, the
desired result follows again from Prop.\ \ref{prop: L stetig maps zero measure to zero}.

Altogether this shows that indeed $\chi_{B_{\varepsilon,A}^{+}(t)}\to\chi_{B_{A}^{+}(t)}\,$
pointwise almost everywhere.
\end{proof}
We are now ready to prove Thm.\ \ref{thm: volume}.
\begin{proof}[Proof of Thm.\ \ref{thm: volume}]
Let $0<t_{1}<t_{2}\leq T$. It suffices to show that 
\[\mathrm{vol}\, B_{A}^{+}(t_{2})\leq\mathrm{vol}\, B_{A}^{+}(t_{1})\:\frac{\mathrm{vol}_{\kappa,\beta}B_{B}^{+}(t_{2})}{\mathrm{vol}_{\kappa,\beta}B_{B}^{+}(t_{1})}.
\]
By Lem.\ \ref{lem:volume convergence} we have
\[
\mathrm{vol}_{\varepsilon}\, B_{\varepsilon,A}^{+}(t)\to\mathrm{vol}\, B_{A}^{+}(t)
\]
for all $t\in(0,T]$. So using Prop.\ \ref{prop: volume for approx}
and letting $\varepsilon\to0$ shows that for all $\eta,\delta>0$
and $B_{\delta,\eta}\subset\Sigma_{\kappa-\delta,\beta+\eta}$ (with
$0<\mathrm{area}_{\kappa-\delta,\beta+\eta}B_{\delta,\eta}<\infty$)
\[
\mathrm{vol}\, B_{A}^{+}(t_{2})\leq\mathrm{vol}\, B_{A}^{+}(t_{1})\:\frac{\mathrm{vol}_{\kappa-\delta,\beta+\eta}B_{B_{\delta,\eta}}^{+}(t_{2})}{\mathrm{vol}_{\kappa-\delta,\beta+\eta}B_{B_{\delta,\eta}}^{+}(t_{1})}
\]
Now by (\ref{eq:volume convergence in comparison space}) there exist
sequences $\delta_{n},\eta_{n}\to0$ such that
\[
\mathrm{vol}_{\kappa-\delta_{n},\beta+\eta_{n}}B_{B_{n}}^{+}(t)\to\mathrm{vol}_{\kappa,\beta}B_{B}^{+}(t)\:\frac{\mathrm{area}_{\kappa-\delta_{n},\beta+\eta_{n}}B_{n}}{\mathrm{area}_{\kappa,\beta}B},
\]
for all $t>0$ which implies
\[
\frac{\mathrm{vol}_{\kappa-\delta_{n},\beta+\eta_{n}}B_{B_{n}}^{+}(t_{2})}{\mathrm{vol}_{\kappa-\delta_{n},\beta+\eta_{n}}B_{B_{n}}^{+}(t_{1})}\to\frac{\mathrm{vol}_{\kappa,\beta}B_{B}^{+}(t_{2})}{\mathrm{vol}_{\kappa,\beta}B_{B}^{+}(t_{1})}.
\]
So $t\mapsto\frac{\mathrm{vol}\, B_{A}^{+}(t)}{\mathrm{vol}_{\kappa,\beta}B_{B}^{+}(t)}$
is indeed nonincreasing on $(0,T]$.
\end{proof}

\section{\label{sec:Applications}Applications}

\subsection{\label{sub:Myers'-theorem-for}Myers' theorem for $\mathcal{C}^{1,1}$-metrics}

We will use the volume comparison result Thm.\ \ref{thm:grmov volume}
to give a proof of Myers' theorem for $\mathcal{C}^{1,1}$-metrics.
\begin{theorem}
\label{thm:myers}Let $\left(M,g\right)$ be a complete $n$-dimensional
Riemannian manifold with $\mathcal{C}^{1,1}$-metric $g$ such that
$\mathbf{Ric}\geq\left(n-1\right)\kappa\, g$ for some $\kappa>0$.
Then $\text{diam}\left(M\right)\leq\frac{\pi}{\sqrt{\kappa}}$.\end{theorem}
\begin{proof}
Let $S_{\kappa}^{n}$ be the $n$-dimensional sphere of radius $\kappa$
with the standard metric, then $S_{\kappa}^{n}$ has constant sectional
curvature $\kappa$ and $\text{diam}\left(S_{\kappa}^{n}\right)=\frac{\pi}{\sqrt{\kappa}}$.
Clearly $\mathrm{vol}_{\kappa}B^{\kappa}(r)$ is constant in $r$
for $r\geq\frac{\pi}{\sqrt{\kappa}}$. By Thm.\ \ref{thm:grmov volume}
this implies that also
\[
r\mapsto\mathrm{vol}B_{p}(r)
\]
is constant for all $r\geq\frac{\pi}{\sqrt{\kappa}}$ and all $p\in M$.
Fix $p$ and assume there exists $q\in M$ with $d(p,q)>\frac{\pi}{\sqrt{k}}$.
Then by continuity of $d(.,p)$ we find a neighborhood $U$ of $q$
with $\mu(U)\neq0$ such that $d(p,q)+1>d(x,p)>\frac{\pi}{\sqrt{\kappa}}$
for all $x\in U$, so $U\subset B_{p}(d(p,q)+1)$ and $U\cap B_{p}(\frac{\pi}{\sqrt{\kappa}})=\emptyset$.
But this shows that
\[
\mathrm{vol}B_{p}(d(p,q)+1)>\mathrm{vol}B_{p}(\frac{\pi}{\sqrt{\kappa}}),
\]
contradicting $r\mapsto\mathrm{vol}B_{p}(r)$ being constant for all
$r\geq\frac{\pi}{\sqrt{\kappa}}$.
\end{proof}
This result is not very surprising since it is known that there are
generalizations of Myers' theorem even for metric measure spaces (see
Cor.\ 2.6 in \citep{sturm_metricMeasureII}). However, these do not
immediately imply Thm.\ \ref{thm:myers} above, because for metric
measure spaces the needed curvature bound is (by necessity) defined
in a different manner from $\mathbf{Ric}\geq\left(n-1\right)\kappa\, g$
in $L_{\mathrm{loc}}^{\infty}$.

\subsection{\label{sub:Hawking's-singularity-theorem}Hawking's singularity theorem
for $\mathcal{C}^{1,1}$-metrics}

We first show a general result concerning geodesic incompleteness of
globally hyperbolic manifolds.
\begin{theorem}
\label{thm:tau less b_kapa_beta (gen hawking)}Assume that $\left(M,g,\Sigma\right)$
(with $g\in\mathcal{C}^{1,1}$) is globally hyperbolic and satisfies
the $CCC(\kappa,\beta)$ condition with either 
\begin{enumerate}
\item $\kappa>0$, $\beta\in\mathbb{R}$,
\item $\kappa=0$, $\beta<0$ or
\item $\kappa<0$, $\beta<0$ such that $\frac{\beta}{(n-1)\,\sqrt{\left|\kappa\right|}}<-1$.
\end{enumerate}
Then $\tau_{\Sigma}(p)\leq b_{\kappa,\beta}<\infty$ for all $p\in I^{+}(\Sigma)$
and $\left(M,g\right)$ is timelike future geodesically incomplete.\end{theorem}
\begin{proof}
First note that for these values of $\kappa$ and $\beta$ we have
$b_{\kappa,\beta}<\infty$ (see Table \ref{tab:Warping-functions-for}),
so $\tau_{\Sigma}(p)\leq b_{\kappa,\beta}$ for all $p\in I^{+}(\Sigma)$
implies $L\left(\gamma\right)\leq b_{\kappa,\beta}$ for all timelike,
f.d.\ geodesics $\gamma$ starting in $\Sigma$, which implies incompleteness
of $M$.

Now assume to the contrary that there exists $p\in I^{+}(\Sigma)$
with $\tau_{\Sigma}(p)>b_{\kappa,\beta}$. We first argue that we
may w.l.o.g.\ assume $p\notin\mathrm{Cut}^{+}(\Sigma)$: By continuity
of $\tau_{\Sigma}$ (see Lem.\ \ref{prop: time sep continuous})
there is a neighborhood $U$ of $p$ such that $\tau_{\Sigma}(q)>b_{\kappa,\beta}$
for all $q\in U$ and since $\mathrm{Cut}^{+}(\Sigma)$ has measure
zero (see \ref{prop:cut locus has measure zero}) but $U$ does not
there exists $\tilde{p}\notin\mathrm{Cut}^{+}(\Sigma)$ with $\tau_{\Sigma}(\tilde{p})>b_{\kappa,\beta}$.

Now, if we have $p\notin\mathrm{Cut}^{+}(\Sigma)$ then, by the same
argument as in the proof of Lem.\ \ref{lem:volume convergence},
there exists a \emph{unique} unit-speed geodesic $\gamma^{p}$ from
$\gamma^{p}(0)\in\Sigma$ to $p$ with $L\left(\gamma^{p}\right)=\tau_{\Sigma}(p)>b_{\kappa,\beta}$
(and this geodesic has to start orthogonally to $\Sigma$ by Lem.\
\ref{prop: time sep continuous}). In particular, $\gamma^{p}$ exists
until at least some $T>\tau_{\Sigma}(p)>b_{\kappa,\beta}$. Let $A$
be a neighborhood of $\gamma^{p}(0)$ in $\Sigma$ such that all unit-speed
geodesics starting in $A$ orthogonally to $\Sigma$ also exist until
at least $T$. We may choose $A$ to be compact with $\mu_{\Sigma}(\partial A)=0$
(e.g.\ as the pre-image of a small, closed ball in $\mathbb{R}^{n-1}$
under a chart of $\Sigma$). 

We now show that there exists a neighborhood $U$ of $p$ such that
for any $q\in\tilde{U}:=U\setminus\mathrm{Cut}^{+}(\Sigma)$ we have
$b_{\kappa,\beta}<\tau_{\Sigma}(q)<T$ (which follows immediately
from continuity of $\tau_{\Sigma}$) and that the \emph{unique} unit-speed
geodesic $\gamma^{q}$ from $\gamma^{q}(0)\in\Sigma$ to $q$ with
$L\left(\gamma^{q}\right)=\tau_{\Sigma}(q)$ satisfies $\gamma^{q}(0)\in A$.
This is done via contradiction in a similar way to case $(1)$ in
the proof of Lem.\ \ref{lem:volume convergence}: Let $p^{+}\in I^{+}(p)$,
then there exists a small neighborhood $U$ of $p$ such that $\gamma^{q}(0)\in J^{-}(p^{+})\cap\Sigma$
for all $q\in\tilde{U}$. Assume there exist $p_{j}\in\tilde{U}$
with $p_{j}\to p$ but $\gamma^{p_{j}}(0)\notin A$. Then $\gamma^{p_{j}}\subset J^{-}(p^{+})\cap J^{+}(J^{-}(p^{+})\cap\Sigma)$
and since this set is compact by Rem.\ \ref{rem: J-capSig compact}
our limit curve Lemma \ref{lem:Limit curve} shows that there exists
$\tilde{\gamma}$ with $p=\tilde{\gamma}(1)$ and $\tilde{\gamma}(0)\neq\gamma^{p}(0)$
and
\[
\tau(\tilde{\gamma}(0),p)=L(\tilde{\gamma})\geq\limsup_{j\to\infty}L(\gamma^{p_{j}})=\limsup_{j\to\infty}\tau_{\Sigma}(p_{j})=\tau_{\Sigma}(p),
\]
by continuity of $\tau_{\Sigma}$ (see Lem.\ \ref{prop: time sep continuous}).
So $\tilde{\gamma}$ is also maximizing the distance between $p$
and $\Sigma$, but this a contradiction since $\tilde{\gamma}\neq\gamma^{p}$
and $\gamma^{p}$ was unique since $p\notin\mathrm{Cut}^{+}(\Sigma)$.

We now apply Thm.\ \ref{thm: volume} to obtain that
\[
t\mapsto\frac{\mathrm{vol}\, B_{A}^{+}(t)}{\mathrm{vol}_{\kappa,\beta}B_{B}^{+}(t)}
\]
is nonincreasing on $(0,T]$. Now the set $\tilde{U}$ from above
satisfies $\mu(\tilde{U})\neq0$ and $\tilde{U}\subset B_{A}^{+}(T)$
but $\tilde{U}\cap B_{A}^{+}(b_{\kappa,\beta})=\emptyset$, hence
$\mathrm{vol}\, B_{A}^{+}(T)>\mathrm{vol}\, B_{A}^{+}(b_{\kappa,\beta})$.
On the other hand, {$\mathrm{vol}_{\kappa,\beta}B_{B}^{+}(t)$}
remains constant in $t$ for $t\geq b_{\kappa,\beta}$ by construction
of the comparison spaces. But then
\[
\frac{\mathrm{vol}\, B_{A}^{+}(T)}{\mathrm{vol}_{\kappa,\beta}B_{B}^{+}(T)}>\frac{\mathrm{vol}\, B_{A}^{+}(b_{\kappa,\beta})}{\mathrm{vol}_{\kappa,\beta}B_{B}^{+}(b_{\kappa,\beta})}
\]
which is a contradiction to $t\mapsto\frac{\mathrm{vol}\, B_{A}^{+}(t)}{\mathrm{vol}_{\kappa,\beta}B_{B}^{+}(t)}$
being nonincreasing on $(0,T]$.\end{proof}
\begin{rem}
If $\mathbf{Ric}\geq\kappa(n-1)g$ with $\kappa>0$ the mean curvature
of $\Sigma$ is irrelevant, hence any globally hyperbolic spacetime
satisfying such a curvature bound is necessarily geodesically incomplete:
By \citep[Thm.~4.5]{clemensGlobHYp} there exists a smooth metric
$g'\succ g$ such that $\left(M,g'\right)$ is globally hyperbolic
as well and by \citep[Thm.~1.1]{bernalSanchez_globHypSplitting}
there exists a smooth, spacelike Cauchy hypersurface $\Sigma$ for
$g'$. This $\Sigma$ is then necessarily acausal (\citep[Lem.~14.29 and 14.42]{ONeill_SRG})
and FCC (see \citep[Rem.~1]{TG}) and thus also a smooth, spacelike,
acausal, FCC Cauchy hypersurface $\Sigma$ for $g$ (by arguments
similar to the ones in Lem.\ \ref{lem: FCC}) and $\tau_{\Sigma}\leq\frac{\pi}{\sqrt{\kappa}}$:
On every compact subset $A\subset\Sigma$ the mean curvature is bounded
from above by some $\beta\in\mathbb{R}$ (and this is all that is
actually needed to show Thm.\ \ref{thm: volume} for this fixed $A$)
and since $b_{\kappa,\beta}\nearrow\frac{\pi}{\sqrt{\kappa}}$ for
$\beta\to\infty$ one arrives at a contradiction by the same construction
as in Thm.\ \ref{thm:tau less b_kapa_beta (gen hawking)}. This even
shows that $L(\gamma)\leq2\frac{\pi}{\sqrt{\kappa}}$ for any timelike
curve $\gamma$ since any inextendible timelike curve must meet $\Sigma$.
Of course, the smooth version of this result is well-known and can
be proven without this detour (\citep[Thm.~11.9]{BEE96}).
\end{rem}
If $\left(M,g\right)$ is not globally hyperbolic, we cannot apply
Thm.\ \ref{thm: volume} directly, but if $\left(M,g,\Sigma\right)$
satisfies $CCC(\kappa,\beta)$ with $\kappa,\beta$ as in Thm.\ \ref{thm:tau less b_kapa_beta (gen hawking)}
and $\Sigma$ is additionally compact we can still use it to prove
compactness of the Cauchy development $D^{+}(\Sigma)$.
\begin{lem}
Let $\left(M,g,\Sigma\right)$ with $g\in\mathcal{C}^{1,1}$ satisfy
$CCC(\kappa,\beta)$ with $\kappa,\beta$ as in Thm.\ \ref{thm:tau less b_kapa_beta (gen hawking)}
and $\Sigma$ compact. If $\left(M,g\right)$ is future geodesically
complete then $D^{+}(\Sigma)$ is relatively compact.\end{lem}
\begin{proof}
By \citep[Thm.~A.22 and Prop.~A.23]{hawkingc11} $D(\Sigma)=D(\Sigma)^{\circ}$
is globally hyperbolic, so we may apply Thm.\ \ref{thm:tau less b_kapa_beta (gen hawking)}
to $(D(\Sigma),g,\Sigma)$, to obtain $\tau_{\Sigma}(p)\leq b_{\kappa,\beta}$
for all $p\in D^{+}(\Sigma)$ and thus $D^{+}(\Sigma)\subset\exp^{N}([0,b_{\kappa,\beta}]\cdot\Sigma)$,
which is compact.
\end{proof}
The case $\kappa=0$, $\beta<0$ of the previous Lemma provides an
alternative proof of Hawking's singularity theorem for $\mathcal{C}^{1,1}$-metrics:
Already in the smooth case the proof of Hawking's singularity theorem
splits into two distinct parts, namely an analytic bit, which shows
relative compactness of $D^{+}(\Sigma)$, and a part using causality
theory. This second part proceeds in the same way whether one deals
with smooth or merely $\mathcal{C}^{1,1}$ metrics, so we will not
repeat it here (see e.g.\ \citep[Thm.~14.55A and 14.55B]{ONeill_SRG}
for the smooth case or \citep[Thm.~1.1]{hawkingc11} for the $\mathcal{C}^{1,1}$
proof%
\footnote{Note that the future convergence in \citep{hawkingc11} is the negative
of the mean curvature and hence it is bounded from below in the assumptions
of \citep[Thm.~1.1]{hawkingc11}.%
}). Thus we obtain:
\begin{theorem}[{\citep[Thm.~1.1]{hawkingc11}}]
\label{thm:C11 hawking for M not glob hyp}Let $\left(M,g,\Sigma\right)$
with $g\in\mathcal{C}^{1,1}$ satisfy $CCC(\kappa,\beta)$ with $\kappa,\beta$
as in Thm.\ \ref{thm:tau less b_kapa_beta (gen hawking)} and $\Sigma$
compact. Then $\left(M,g\right)$ is future geodesically incomplete.
\end{theorem}
There seem to be several advantages of this new approach. First, it
illustrates the interdependence of the two curvature bounds $\kappa$
and $\beta$ very nicely (see conditions $(1)$ to $(3)$ in Thm.\
\ref{thm:tau less b_kapa_beta (gen hawking)}): The parameter $\beta$
describes the initial focusing ($\beta<0$) or defocusing ($\beta>0$)
of geodesics emanating orthogonally to $\Sigma$ (looking at the comparison
manifolds in Table \ref{tab:Warping-functions-for} we see that $\left|f_{\kappa,\beta}\right|$
is initially decreasing if $\beta<0$ and increasing if $\beta>0$
and by the formula for the areas in the proof of (\ref{eq:area convergence in comparison space})
the same remains true for $\mathrm{area}_{\kappa,\beta}S_{A}^{+}(t)$),
while $\kappa$ describes a global focusing ($\kappa>0$) or defocusing
($\kappa<0$) effect for timelike geodesics. Depending on their relative
strength there exists a time $t=b_{\kappa,\beta}$ where $f_{\kappa,\beta}$
becomes zero (and the comparison manifold becomes singular) or not.
By the volume comparison Theorem \ref{thm: volume} (and its application
in Thm.\ \ref{thm:C11 hawking for M not glob hyp}) this time gives
a universal bound on the maximal time of existence of geodesics starting
orthogonally to $\Sigma$ in globally hyperbolic manifolds satisfying
the respective curvature bounds. While of course this behavior is
also present in the Rauchaudhuri argument used in \citep{hawkingc11}
(and for the smooth case in e.g.\ \citep{Seno1}) and an analogous
argument would also suffice to show cases $(1)$ and $(3)$ from Thm.\
\ref{thm:tau less b_kapa_beta (gen hawking)}, it seems that it is
somewhat more explicit in the comparison treatment given here.

Second, while the proof of Thm.\ \ref{thm: volume} again relies
on approximation arguments, the volume comparison result itself now
provides a tool which works directly in $\mathcal{C}^{1,1}$ and allows
us to prove other important results (e.g., Thm.\ \ref{thm:tau less b_kapa_beta (gen hawking)}
and Thm.\ \ref{thm:C11 hawking for M not glob hyp}) without returning
to the smooth case. 

And, perhaps most importantly, the volume comparison theorem Thm.\
\ref{thm: volume} itself is of considerable interest: As already
pointed out by the authors of \citep{TG}, their results are remarkably
close to the corresponding Riemannian ones and thus might lend themselves
to generalizations of curvature bounds to even lower regularity, a
hope that may be strengthened by the $\mathcal{C}^{1,1}$ version
of their volume comparison result (\citep[Thm.~9]{TG}) proven here.

\subsection*{Acknowledgment}
The author is grateful to James D.E.\ Grant, Michael Kunzinger and
Roland Steinbauer for helpful discussions and feedback to drafts of
this work. This work was supported by the Austrian Science Fund (FWF), project numbers P25326 and P28770. The author is also the recipient of a DOC Fellowship of the Austrian Academy of Sciences at the Institute of Mathematics at the University of Vienna.

\appendix

\section{Some results from measure theory}

To show the measurability of the cut function in Lem.\ \ref{lem: cut function measurable}
we need some tools from measure theory, the main one being the measurable
projection theorem (see \citep[Thm.~III.23]{CVmeasuremultifunc}):
\begin{theorem}
[Measurable Projection] \label{thm:measure projection}Let $\left(\Omega,\mathcal{A}\right)$
be a measurable space and $S$ a Suslin space. If $G$ is measurable
in the product $\sigma$-algebra $\mathcal{A}\otimes\mathcal{B}(S)$
(where $\mathcal{B}(S)$ denotes the Borel-$\sigma$-algebra of $S$),
then its projection $\mathrm{pr}_{\Omega}(G)\subset\Omega$ is universally
measurable.
\end{theorem}
This statement uses the following definitions (see \citep[Def.~III.17 and III.21]{CVmeasuremultifunc}):
\begin{defn}
[Suslin and Polish spaces] A Suslin space is a Hausdorff topological
space that is the continuous image of a Polish space. A Polish space
is a separable completely metrizable topological space.\end{defn}
\begin{example}
\label{ex: R suslin}Clearly, $\mathbb{R}$ is Polish, hence also
Suslin.\end{example}
\begin{defn}
[Universal $\sigma$-algebra]Let $\left(\Omega,\mathcal{A}\right)$
be a measurable space. Given any finite measure $\mu$ we denote the
completion of $\mathcal{A}$ with respect to $\mu$ by $\mathcal{A}_{\mu}$.
Then the universal $\sigma$-algebra $\hat{\mathcal{A}}$ is defined
as 
\[
\hat{\mathcal{A}}:=\bigcap_{\mu\,\mathrm{finite}}\mathcal{A}_{\mu}.
\]
\end{defn}
\begin{rem}
If $\mu$ is a $\sigma$-finite measure on $\left(\Omega,\mathcal{A}\right)$
then there exists an equivalent (i.e.\ having the same zero-measure
sets) measure $\tilde{\mu}$ that is finite. So one has
\[
\hat{\mathcal{A}}=\bigcap_{\mu\,\sigma-\mathrm{finite}}\mathcal{A}_{\mu}.
\]
This shows that any universally measurable set is measurable with
respect to every complete $\sigma$-finite measure $\mu$ on $\left(\Omega,\mathcal{A}\right)$.
\end{rem}
This allows us to show the following:
\begin{prop}
\label{prop:measurable sup}Given a measurable space $\left(\Omega,\mathcal{A}\right)$,
a Suslin space $S$, a measurable function $f:\Omega\times S\to\mathbb{R}$
(w.r.t.\ to the product $\sigma$-algebra $\mathcal{A}\otimes\mathcal{B}(S)$
on $\Omega\times S$ and the Borel-$\sigma$-algebra on $\mathbb{R}$)
and a set-valued map $F:\Omega\to\mathcal{P}(S)$ with $\mathrm{graph}(F):=\left\{ \left(\omega,s\right)\in\Omega\times S:\, s\in F(\omega)\right\} \in\mathcal{A}\otimes\mathcal{B}(S)$,
one has that the map $f^{*}:\Omega\to\bar{\mathbb{R}}$ defined by
\[
f^{*}(\omega):=\sup\left\{ f(\omega,s):\, s\in F(\omega)\right\} 
\]
is measurable with respect to the universal $\sigma$-algebra $\hat{\mathcal{A}}$
(and the Borel-$\sigma$-algebra on $\bar{\mathbb{R}}$).\end{prop}
\begin{proof}
It clearly suffices to show that $\left\{ \omega:\, f^{*}(\omega)>r\right\} \in\hat{\mathcal{A}}$
for all $r\in\mathbb{R}$. This follows immediately from
\[
\left\{ \omega:\, f^{*}(\omega)>r\right\} =\mathrm{pr}_{\Omega}\left(\mathrm{graph}(F)\cap\left\{ (\omega,s):\, f(\omega,s)>r\right\} \right),
\]
measurability of $f$, $\mathrm{graph}(F)\in\mathcal{A}\otimes\mathcal{B}(S)$
and the measurable projection theorem \ref{thm:measure projection}.
\end{proof}
Another statement we need concerns itself with the measurability of
the graph of a measurable function (and can, e.g., be found in \citep[Prop.~3.1.21]{Srivastava_borelSets})
\begin{prop}
\label{prop: graph is measurable}Let $\left(\Omega,\mathcal{A}\right)$
be a measurable space and $X$ a second countable topological space
satisfying the $T_{1}$ separation axiom (for every pair of distinct
points there exists a neighborhood for each that does not contain
the other) with Borel-$\sigma$-algebra $\mathcal{B}(X)$. If $f:\Omega\to X$
is measurable, then $\mathrm{graph}(f)\in\mathcal{A}\otimes\mathcal{B}(X)$.\end{prop}
\begin{proof}
Let $U_{n}$ be a countable basis of open sets, then
\[
y\neq f(x)\Longleftrightarrow\exists n:\, f(x)\in U_{n}\mathrm{\, and\,}y\notin U_{n}.
\]
So we have
\[
\mathrm{graph}(f)=\left[\bigcup_{n}f^{-1}(U_{n})\times U_{n}^{c}\right]^{c},
\]
hence it is measurable.
\end{proof}
Furthermore, if the graph of a function between two $\sigma$-finite
measure spaces is measurable (and points have measure zero), it has
measure zero:
\begin{prop}
\label{prop: graph has measure zero}Let $\left(\Omega_{1},\mathcal{A}_{1},\mu_{1}\right)$
and $\left(\Omega_{2},\mathcal{A}_{2},\mu_{2}\right)$ be two $\sigma$-finite
measure spaces such that $\mu_{2}(\left\{ x\right\} )=0$ for all
$\left\{ x\right\} \in\mathcal{A}_{2}$. Assume $f:\Omega_{1}\to\Omega_{2}$
has measurable graph, i.e.\ $\mathrm{graph}(f)\in\mathcal{A}_{1}\otimes\mathcal{A}_{2}$,
then $\mu_{1}\otimes\mu_{2}(\mathrm{graph}(f))=0$.\end{prop}
\begin{proof}
We apply Fubini's theorem to the characteristic function $\chi_{\mathrm{graph}(f)}$.
This gives that for any $x\in\Omega_{1}$ the functions
\[
\chi_{\mathrm{graph}(f)}(x,.)=\chi_{\{f(x)\}}
\]
from $\Omega_{2}\to\mathbb{R}$ are measurable, hence $\left\{ f(x)\right\} \in\mathcal{A}_{2}$
for all $x$ and so by assumption on $\mu_{2}$ we have $\mu_{2}\left(\left\{ f(x)\right\} \right)=0$.
But then Fubini gives
\[
\mu_{1}\otimes\mu_{2}(\mathrm{graph}(f))=\int_{\Omega_{1}}\int_{\Omega_{2}}\chi_{\{f(x)\}}(y)d\mu_{2}(y)d\mu_{1}(x)=0,
\]
proving the claim.
\end{proof}
Finally, we will state a result concerning images of sets of measure
zero under Lipschitz continuous functions on $\mathbb{R}^{n}$ (which
can be found, e.g., in \citep[Prop.~3.2]{Kahn_globAna} for differentiable
maps, but the proof only uses the Lipschitz property) that is needed
in various proofs of this work.
\begin{prop}
\label{prop: L stetig maps zero measure to zero}Let $f:\mathbb{R}^{n}\to\mathbb{R}^{n}$
be Lipschitz continuous. If $A\subset\mathbb{R}^{n}$ has (Lebesgue-)measure
zero, then $f(A)\subset\mathbb{R}^{n}$ has (Lebesgue-)measure zero
as well.
\end{prop}
\bibliographystyle{amsplain}
\bibliography{bibtex_BS2}

\end{document}